\documentclass{amsart}
\pdfoutput=1

\usepackage{amsmath, amsthm, amssymb, mathrsfs, mathtools, graphicx}
\usepackage{hyperref}

\newcommand{\st}{~|~}
\newcommand{\mst}{~\middle|~}
\renewcommand{\phi}{\varphi}
\renewcommand{\epsilon}{\varepsilon}
\renewcommand{\emptyset}{\varnothing}
\renewcommand{\setminus}{\smallsetminus}
\newcommand{\e}{\epsilon}
\newcommand{\dee}{\partial}
\newcommand{\R}{\mathbb{R}}
\newcommand{\C}{\mathbb{C}}
\newcommand{\Z}{\mathbb{Z}}
\newcommand{\N}{\mathbb{N}}
\newcommand{\Q}{\mathbb{Q}}
\renewcommand{\Re}{\mathrm{Re}}
\renewcommand{\Im}{\mathrm{Im}}
\DeclareMathOperator{\GL}{\mathrm{GL}}

\newcommand{\1}{^{-1}}
\DeclareMathOperator{\modulus}{\mathrm{mod}}
\DeclareMathOperator{\hei}{\mathrm{height}}
\DeclareMathOperator{\width}{\mathrm{width}}
\DeclareMathOperator{\slope}{\mathrm{slope}}
\DeclareMathOperator{\supp}{\mathrm{supp}}
\renewcommand{\Tilde}{\widetilde}
\DeclareMathOperator{\hocolim}{\mathrm{hocolim}}
\DeclareMathOperator{\colim}{\mathrm{colim}}
\DeclareMathOperator{\sign}{\mathrm{sign}}

\newtheorem{thm}{Theorem}
\numberwithin{thm}{section}
\newtheorem{prop}[thm]{Proposition}
\newtheorem{lem}[thm]{Lemma}

\theoremstyle{definition}
\newtheorem{defn}[thm]{Definition}

\newtheorem{rem}[thm]{Remark}

\title[$L^\infty$-isodelaunay decomposition of strata]{The $L^\infty$-isodelaunay decomposition of strata of abelian differentials}
\author{Bradley Zykoski}
\address{Department of Mathematics, University of Michigan, Ann Arbor, MI}
\email{\href{mailto:zykoskib@umich.edu}{zykoskib@umich.edu}}

\begin{document}
\maketitle
\thispagestyle{empty}

\begin{abstract}
We study the decomposition of a stratum $\mathcal H(\kappa)$ of abelian differentials into regions of differentials that share a common $L^\infty$-Delaunay triangulation. In particular, we classify the infinitely many adjacencies between these isodelaunay regions, a phenomenon whose observation is attributed to Filip in work of Frankel. This classification allows us to construct a finite simplicial complex with the same homotopy type as $\mathcal H(\kappa)$, and we outline a method for its computation. We also require a stronger equivariant version of the traditional Nerve Lemma than currently exists in the literature, which we prove.
\end{abstract}

\setcounter{tocdepth}{1}
\tableofcontents
\section{Introduction}
Let $S_g$ be a surface of genus $g$, and let $\Sigma \subset S_g$ be a finite subset of size $n$. The topology of the moduli space $\mathcal M_{g,n}$ of Riemann surface structures on $S_g$ with marked points in $\Sigma$ is a rich subject with much remaining to discover. One of the landmark results in this direction was Harer's computation in \cite{harer} that the virtual cohomological dimension of $\mathcal M_{g,0}$ is $4g-5$. This involved the construction of a simplicial complex $Y^0$ whose cells were indexed by topological substructures of $(S_g, \Sigma)$ called arc systems, such that $\mathcal M_{g,n}$ is homotopy equivalent to the quotient of $Y^0$ by a mapping class group. More than three decades later, it has recently been shown the cohomology only one degree below the virtual cohomological dimension is very rich: Theorem 1.1 of \cite{chan-galatius-payne} states that $\dim H^{4g-6}(\mathcal M_{g,0};\Q)$ grows at least exponentially in $g$.

Given a partition $\kappa = (k_1,\dots,k_n)$ of $2g-2$, the stratum $\mathcal H(\kappa)$ of abelian differentials is the moduli space of pairs $(X, \omega)$ where $X$ is a Riemann surface structure on $S_g$ and $\omega$ is an abelian differential on $X$ with zeroes of orders $\kappa$ at the points $\Sigma$. Systematic study of the topology of $\mathcal H(\kappa)$ began much more recently than that of $\mathcal M_{g,n}$. The first major results in this direction were Theorems 1 and 2 of \cite{kontsevich-zorich}, which classified the connected components of $\mathcal H(\kappa)$. Later major results include Theorem A of \cite{calderon-salter}, which characterizes the image of a canonical monodromy representation of the orbifold fundamental group of a component of $\mathcal H(\kappa)$, and Theorem 1.3 of \cite{costantini-moeller-zachhuber}, which provides a computable recursion for the orbifold Euler characteristic of $\mathcal H(\kappa)$. It is shown in \cite{looijenga-mondello} that most strata in genus 3, and all hyperelliptic components of strata, are orbifold $K(G,1)$ spaces, but it is unknown in general whether components of $\mathcal H(\kappa)$ are orbifold $K(G,1)$ spaces. Further work on the orbifold fundamental groups of strata includes \cite{hamenstaedt}, \cite{calderon}, and \cite{calderon-salter-spin}, and further related results of an algebro-geometric nature include \cite{mondello}, \cite{bainbridge-et-al}, and \cite{chen}. Study of the topology of closely-related moduli spaces has also been developing in works such as \cite{lanneau}, \cite{boissy-ends}, \cite{boissy}, \cite{bainbridge-et-al-k-diff}, \cite{chen-gendron}, and \cite{apisa-bainbridge-wang}. No systematic calculations of the dimensions of nontrivial cohomology groups of $\mathcal H(\kappa)$ are currently known, nor are any presentations for the fundamental groups of the components of $\mathcal H(\kappa)$.

We introduce a simplicial complex $\mathcal N(\mathscr D_{\mathrm{finite}}(\kappa))$ whose cells are indexed by topological substructures of $(S_g, \Sigma)$ called veering triangulations, such that we have the following theorem.
\begin{thm}\label{intro main}
The stratum $\mathcal H(\kappa)$ is homotopy equivalent to the quotient of $\mathcal N(\mathscr D_{\mathrm{finite}}(\kappa))$ by the mapping class group $\mathrm{MCG}(S_g,\Sigma)$. After barycentric subdivision, this quotient is a finite simplicial complex $\mathcal I(\kappa)$.
\end{thm}
\noindent We call $\mathcal I(\kappa)$ the \emph{isodelaunay complex}. In Section \ref{complex section} we prove Theorem \ref{intro main} and outline how $\mathcal I(\kappa)$ may be computed explicitly via a method that requires only linear programming and the combinatorics of surface triangulations. Further work is presently underway to implement this computation of $\mathcal I(\kappa)$, building off of the software package \texttt{veerer} \cite{veerer} designed to handle computations with veering triangulations. Such a computation is anticipated to produce calculations of Betti numbers and fundamental group presentations for $\mathcal H(\kappa)$ when $\kappa$ is small. It is hoped that these calculations, as well as further theoretical understanding of the isodelaunay complex, will shed light on the topology of $\mathcal H(\kappa)$ in analogy with Harer's complex for $\mathcal M_{g,n}$.

The complex $\mathcal N(\mathscr D_{\mathrm{finite}}(\kappa))$ is constructed by studying $L^\infty$-Delaunay triangulations of abelian differentials. Such triangulations were first introduced in \cite{gueritaud}, and were further investigated in \cite{frankel-cat}. They can be seen as a special case of veering triangulations, which were first introduced in the context of pseudo-Anosov mapping tori in \cite{agol}. Whereas we use $L^\infty$-Delaunay triangulations to analyze the topology of $\mathcal H(\kappa)$, they were used in \cite{gueritaud} and \cite{frankel-cat} to study pseudo-Anosov mapping classes and the Teichm\"uller geodesic flow; see also \cite{minsky-taylor} and \cite{bell-et-al} for the role of veering triangulations in this direction. We concern ourselves with subsets of $\mathcal H(\kappa)$ of differentials that share the same $L^\infty$-Delaunay triangulation, which we call \emph{isodelaunay regions}. In \cite{frankel-cat}, Frankel relates an observation made by Simion Filip on the complicated intersection patterns of these regions, a phenomenon that we call the \emph{infinite adjacency phenomenon}, and anticipates a finitary simplification of these complicated patterns; see Remark \ref{frankel's remark}. We provide such a simplification in Theorem \ref{classification result}, and in Section \ref{complex section} we use this to construct the complex $\mathcal N(\mathscr D_{\mathrm{finite}}(\kappa))$ and prove Theorem \ref{intro main}. Isodelaunay regions form a covering of $\mathcal H(\kappa)$ by finitely many coordinate charts, and our concrete understanding of these regions opens the door to the study of invariant subvarieties of $\mathcal H(\kappa)$ via explicit coordinate representations.

\subsubsection*{Overview of the main argument}

The infinite adjacency phenomenon and its role as an obstruction to constructing a finite simplicial complex homotopy equivalent to $\mathcal H(\kappa)$ are elucidated by the following example. We remark that this example is highly analogous to the case of $\mathcal H(0)$.

Consider the covering of $\R^3$ by closed sets $A_i = [i, i+1] \times \R \times [0,\infty)$ and $B_i = \R \times [i, i+1] \times (-\infty, 0]$ for $i \in \Z$. This covering exhibits an infinite adjacency phenomenon: each $A_i$ intersects every $B_j$ nontrivially. The action of $G = \Z$ on $\R^3$ via $n.(x,y,z) = (x+n,y+n,z)$ satisfies $n.A_i = A_{i+n}$ and $n.B_i = B_{i+n}$. The covering $\{A_i, B_i\}_{i \in \Z}$ is therefore partitioned into two $G$-orbits, and hence $\R^3/G$ is covered by the two sets $A = GA_0$ and $B = GB_0$. A remnant of the infinite adjacency phenomenon persists: the common boundary of $A$ and $B$ is the union of infinitely many $G$-orbits of sets $G(A_i \cap B_j)$. In the case of a more well-behaved covering of a $G$-space $X = \bigcup_{i \in I} U_i$, the nerve $\mathcal N(\{U_i\}_{i \in I})$ gives rise to a finite simplicial complex $\mathcal N(\{U_i\}_{i \in I})/G$ homotopy equivalent to the quotient $X/G$. Our remnant of the infinite adjacency phenomenon obstructs this in our example, since there is an edge of $\mathcal N(\{A_i, B_i\}_{i \in \Z})/G$  for each of the infinitely many $G(A_i \cap B_j)$.

The stratum $\mathcal H(\kappa)$ can be realized as the quotient of a space $\mathcal H_{\mathrm{marked}}(\kappa)$ by a mapping class group $G = \mathrm{MCG}(S_g, \Sigma)$, and this space admits a covering $\mathscr D(\kappa)$ by the closures of isodelaunay regions. Again we have an infinite adjacency phenomenon: the nerve $\mathcal N(\mathscr D(\kappa))$ is not locally finite. Hence the quotient $\mathcal N(\mathscr D(\kappa))/G$ is not a finite simplicial complex. To obtain a finite complex homotopy equivalent to $\mathcal H(\kappa)$, we find a $G$-invariant subset $\mathcal T(\kappa) \subset \mathcal H_{\mathrm{marked}}(\kappa)$ such that $\mathcal H_{\mathrm{marked}}(\kappa) \setminus \mathcal T(\kappa)$ is homotopy equivalent to $\mathcal H_{\mathrm{marked}}(\kappa)$, and the nerve of the induced covering $\mathscr D_{\mathrm{finite}}(\kappa)$ of $\mathcal H_{\mathrm{marked}}(\kappa) \setminus \mathcal T(\kappa)$ is locally finite. The sets of $\mathscr D_{\mathrm{finite}}(\kappa)$ have nontrivial stabilizers in $G$, and so we take the second barycentric subdivision $\mathcal N(\mathscr D_{\mathrm{finite}}(\kappa))''$ to ensure that the quotient $\mathcal I(\kappa) = \mathcal N(\mathscr D_{\mathrm{finite}}(\kappa))''/G$ is a simplicial complex. Finally, we use our equivariant Nerve Lemma to conclude $\mathcal H(\kappa) \simeq \mathcal N(\mathscr D_{\mathrm{finite}}(\kappa))''/G$.

\subsubsection*{Outline of the paper}

In Section \ref{isodelaunay section}, we introduce our primary geometric tool, the $L^\infty$-Delaunay triangulation, and establish the basic properties of isodelaunay regions. In Section \ref{decomposition section}, we discuss the structural properties of the decomposition of $\mathcal H(\kappa)$ into such regions. Section \ref{infinite adjacency section} is the core technical heart of the paper, in which we introduce and resolve the problem that is the infinite adjacency phenomenon. In Section \ref{complex section}, we use our classification to produce the finite simplicial complex $\mathcal I(\kappa)$ that is homotopy equivalent to $\mathcal H(\kappa)$, and we outline a method for its explicit computation. In the appendix, we formulate and prove a more general form of the traditional Nerve Lemma than currently exists in the literature, which we use to establish the homotopy equivalence between $\mathcal I(\kappa)$ and $\mathcal H(\kappa)$.
\subsubsection*{Acknowledgments}
The author would like to thank his advisor Alex Wright for his guidance throughout the course of this project, as well as Sayantan Khan, Saul Schleimer, and Christopher Zhang for helpful conversations. The author is grateful to have been partially supported by NSF Grant DMS 1856155.

\section{Isodelaunay regions}\label{isodelaunay section}

We begin this section by recording the definitions and notation that will be fundamental to our study of strata of abelian differentials. Let $g \geq 1$, let $S_g$ be a surface of genus $g$, and let $\Sigma = \{x_1,\dots,x_n\} \subset S_g$ be a finite subset. Recall the following definition.
\begin{defn}
The structure of a \emph{translation surface} on $(S_g,\Sigma)$ is an atlas of charts $U \subset S_g\setminus\Sigma \to \C$ such that
\begin{enumerate}
\item The changes-of-coordinates are Euclidean translations $z \mapsto z+\alpha$, where $\alpha \in \C$, and hence induce a Euclidean metric on $S_g \setminus \Sigma$,
\item The points $x_i$ of $\Sigma$ are cone singularities of the metric completion of this Euclidean metric to $S_g$.
\end{enumerate}
\end{defn}

\begin{rem}
Just as the translation surface structure induces a Euclidean metric on $S_g \setminus \Sigma$ via pullback from $\C$, it also induces a metric by pulling back the metric given by the $L^\infty$ norm $\|\cdot\|_\infty$. Open balls of this induced metric are the images of what we call $L^\infty$\emph{-squares} in the following subsection.
\end{rem}

The angles of the Euclidean cone singularities are $2\pi(k_i + 1)$, where the $k_i$ are nonnegative integers whose sum is $-\chi(S_g) = 2g-2$. If $k_i = 0$, we say that $x_i$ is a \emph{marked point} of the translation surface.

\begin{defn}[Stratum of marked translation surfaces]
Let $\kappa = (k_1,\dots,k_n)$ be any tuple of nonnegative integers whose sum is $2g-2$. Then we denote by $\mathcal H_\mathrm{marked}(\kappa)$ the space of translation surface structures, up to isotopy rel $\Sigma$, on $S_g$ with a cone singularity of angle $2\pi(k_i+1)$ at $x_i$ for each $1 \leq i \leq n$.
\end{defn}

A Riemann surface structure $X$ on $S_g$ with an abelian differential $\omega \in H^{1,0}(X)$ whose set of zeroes\footnote{We consider marked points as ``zeroes of order 0'' for $\omega$.} is $\Sigma$ induces a translation surface structure $M$ on $(S_g,\Sigma)$, and it is not hard to see that all translation surface structures arise in this way. We may therefore use the terms ``abelian differential'' and ``translation surface'' interchangeably, and write $M = (X,\omega)$.

\begin{defn}
By the \emph{universal cover} $\Tilde M \twoheadrightarrow M$ of a translation surface $M = (X,\omega)$, we mean the universal cover $p:\Tilde X \twoheadrightarrow X$ endowed with the differential form $\Tilde \omega \coloneqq p^*\omega$. Equivalently, $\Tilde M$ is the universal cover $p:\Tilde{S_g} \twoheadrightarrow S_g$ endowed with the metric induced by pulling back, along $p$, the Euclidean metric on $S_g$ given by $M$.
\end{defn}

\begin{defn}
Let $M = (X,\omega) \in \mathcal H_{\mathrm{marked}}(\kappa)$ and let $\gamma \in H_1(S_g,\Sigma; \Z)$. Then $\gamma$ and $M$ determine a complex number via
\[
\gamma(M) \coloneqq \int_\gamma \omega \in \C.
\]
Let $\mathscr B = \{\gamma_1,\dots,\gamma_{2g+n-1}\}$ be any basis for $H_1(S_g,\Sigma;\Z)$. Then we have a map $\Phi:\mathcal H_{\mathrm{marked}}(\kappa) \to \C^{2g+n-1}$ given by
\[
\Phi (M)_j \coloneqq \gamma_j(M) = \int_{\gamma_j}\omega
\]
for all $1 \leq j \leq 2g+n-1$.
The map $\Phi$, called the \emph{period coordinate map} is a local homeomorphism. An open subset of $\mathcal H_{\mathrm{marked}}(\kappa)$ on which $\Phi$ is a homeomorphism onto its image is called a \emph{period coordinate chart}.
\end{defn}

The space $\mathcal H_{\mathrm{marked}}(\kappa)$ is an auxiliary tool for studying the ordinary stratum of translation surfaces.

\begin{defn}[Stratum of translation surfaces]
Let $\kappa = (k_1,\dots,k_n)$ be any tuple of nonnegative integers whose sum is $2g-2$. Then we denote by $\mathcal H(\kappa)$ the space of translation surface structures, up to isometry, on $S_g$ with a cone singularity of angle $2\pi(k_i+1)$ at $x_i$ for each $1 \leq i \leq n$.
\end{defn}

\begin{defn}[Mapping class group]
We denote by $\mathrm{MCG}(S_g, \Sigma)$ the group of orientation-preserving homeomorphisms $\phi: S_g \xrightarrow{\sim} S_g$ satisfying $\phi(\Sigma) = \Sigma$, considered up to isotopy rel $\Sigma$.
\end{defn}

We have an infinite-degree orbifold covering map $\mathcal H_{\mathrm{marked}}(\kappa) \twoheadrightarrow \mathcal H(\kappa)$ given by forgetting the isotopy marking. Note that this need not be a universal covering. Observe that $\mathrm{MCG}(S_g,\Sigma)$ acts on $\mathcal H_{\mathrm{marked}}(\kappa)$ by re-marking, and that this covering map is the quotient projection $\mathcal H_{\mathrm{marked}}(\kappa) \twoheadrightarrow \mathcal H_{\mathrm{marked}}(\kappa) / \mathrm{MCG}(S_g, \Sigma) = \mathcal H(\kappa)$.

\begin{rem}\label{induced linear structure}
The changes-of-coordinates between period coordinate charts on $\mathcal H_{\mathrm{marked}}(\kappa)$ are induced by changes-of-basis for $H_1(S_g, \Sigma; \Z)$, and hence lie in $\GL_{2g+n-1}(\Z)$, and thereby endow $\mathcal H_{\mathrm{marked}}(\kappa)$ with the structure of an integral affine manifold. The re-marking action of $\mathrm{MCG}(S_g,\Sigma)$ acts via integral changes-of-basis in local period coordinates, and hence respects the integral affine structure. The covering map $\mathcal H_{\mathrm{marked}}(\kappa) \twoheadrightarrow \mathcal H(\kappa)$ thus endows $\mathcal H(\kappa)$ with the structure of an integral affine orbifold.
\end{rem}

\begin{rem}\label{gl2r action}
The group $\GL_2^+(\R)$ acts on $\mathcal H_{\text{marked}}(\kappa)$ and $\mathcal H(\kappa)$ as follows. Given $A \in \GL_2^+(\R)$ and a translation surface $M$, if $\phi:U \to \C$ is a coordinate chart for $M$, then $A \circ \phi$ is a coordinate chart for $AM$.
\end{rem}

\setcounter{subsection}{1}
\subsection{The $L^\infty$-Delaunay triangulation}

\begin{defn}
An $L^\infty$\emph{-square} in a translation surface $M = (X, \omega)$ is an immersion
\[
S: (0,h)^2 \to X, \quad \quad h > 0,
\]
such that $S^*\omega = dz$. In particular, no singularities of $M$ lie in the image of $S$. We write $\mathrm{height}(S) = h$. An $L^\infty$-square is \emph{maximal} if it is not properly contained in any other $L^\infty$-square.

We will find it convenient to refer to the continuous extension $\overline{S}: [0,h]^2 \to X$. We will say that a singularity $x$ \emph{lies on the boundary} of an $L^\infty$-square $S$ if it lies in the image of $\overline{S}$.
\end{defn}

\begin{rem}\label{lifting remark}
Note that it may happen that $\overline{S}\1(x)$ consists of more than one point. We will therefore often find it convenient to refer to a lift $\Tilde{S}: (0,h) \to \Tilde M$ of $S$ to the universal cover $\Tilde M$ when we want to discuss singularities that meet the boundary of an $L^\infty$-square $S$, because the extension $\overline{\Tilde{S}}$ is an embedding. As an embedding, we will often identify it with its image.
\end{rem}

\begin{defn}
We consider paths on $S_g$ that are geodesic with respect to the Euclidean metric induced by $M$.
A \emph{saddle connection} is a geodesic path whose endpoints lie in $\Sigma$ and whose interior contains no points of $\Sigma$. We say that a geodesic path $\gamma$ is \emph{inscribed} in an $L^\infty$-square $S:(0,h)^2 \to M$ if $\gamma$ is the image under $\overline{S}$ of a line segment with endpoints on the boundary of $[0,h]^2$.
\end{defn}

\begin{defn}\label{linfty delaunay triangle defn}
Let $T$ be a triangle on $M$ whose edges are geodesic paths. We say that an $L^\infty$-square $S:(0,h)^2 \to M$ is a \emph{circumsquare} of $T$ if $T$ is the image under $\overline{S}$ of a triangle inscribed in $[0,h]^2$. We say that $T$ is an $L^\infty$-\emph{Delaunay triangle} if it has a circumsquare $S$, and if
\begin{enumerate}
\item\label{local delaunay condition} The only points $x \in [0,h]^2$ such that $\overline S(x) \in \Sigma$ are the vertices of the inscribed triangle.
\item\label{local veering condition} The edges of $T$ are neither horizontal nor vertical.
\end{enumerate}
Note that the edges of an $L^\infty$-Delaunay triangle are saddle connections, since their interiors lie in the image of $S$, and hence cannot not contain any points of $\Sigma$.
\end{defn}

\begin{defn}\label{delaunay triangle defn}
Let $\Delta$ be a triangulation of $S_g$ whose set of vertices is $\Sigma$. We say that $\Delta$ is an $L^\infty$-\emph{Delaunay triangulation} of a translation surface $M$ if every triangle of $\Delta$ is $L^\infty$-Delaunay on $M$. We denote by $E(\Delta)$ the set of edges of a triangulation.
\end{defn}

\begin{defn}\label{delaunay defn}
A translation surface $M$ is $L^\infty$\emph{-generic} if
\begin{enumerate}
\item\label{delaunay condition} No maximal $L^\infty$-square in the universal cover $p:\Tilde M \twoheadrightarrow M$ has more than 3 singularities on its boundary,
\item\label{veering condition} No saddle connection in $\Tilde M$ lies on the boundary of a maximal $L^\infty$-square.
\end{enumerate}
We denote by $\mathcal G(\kappa) \subset \mathcal H(\kappa)$ and $\mathcal G_{\mathrm{marked}}(\kappa) \subset \mathcal H_{\mathrm{marked}}(\kappa)$ the subspaces of $L^\infty$-generic surfaces. Note that $\mathcal G_{\mathrm{marked}}(\kappa)$ is the preimage of $\mathcal G(\kappa)$ under the covering $\mathcal H_{\mathrm{marked}}(\kappa) \twoheadrightarrow \mathcal H(\kappa)$.
\end{defn}

\begin{defn}\label{isodelaunay defn}
An \emph{isodelaunay region} $\mathcal D^\circ$ of $\mathcal H_{\mathrm{marked}}(\kappa)$ is a connected component of $\mathcal G_{\mathrm{marked}}(\kappa)$.
\end{defn}

\begin{rem}
The analogous triangulation for the Euclidean $L^2$-metric, often simply called the Delaunay triangulation, is more traditionally studied. However, it remains an open question whether the analogous $L^2$-isodelaunay regions are even contractible, whereas the $L^\infty$-isodelaunay regions are known to be convex polytopes when expressed in period coordinates (Theorem \ref{polytope theorem}). We use this fact in many places besides its implication that the regions are contractible. Whether or not they turn out to be contractible, $L^2$-isodelaunay regions are not polytopes when expressed in period coordinates, because the inequalities that must be locally satisfied in order for a family of translation surfaces to maintain the same $L^2$-Delaunay triangulation are of higher order.
\end{rem}

In the remainder of this subsection, we discuss the basic properties of $L^\infty$-generic surfaces.

\begin{lem}\label{open density}
The set $\mathcal G(\kappa)$ is an open dense subset of $\mathcal H(\kappa)$.
\end{lem}

\begin{proof}
Clearly $\mathcal G(\kappa)$ is an open subset of $\mathcal H(\kappa)$. The set of translation surfaces having no horizontal or vertical saddle connections is a dense subspace of $\mathcal H_{\mathrm{marked}}(\kappa)$, and therefore so too is the set of translation surfaces satisfying condition (\ref{veering condition}) of Definition \ref{delaunay defn}. Furthermore, if $4$ or more singularities of a translation surface lie on a common maximal $L^\infty$-square, then this imposes an equality on the lengths and widths of the saddle connections joining these singularities. Therefore also the set of translation surfaces satisfying condition (\ref{delaunay condition}) of Definition \ref{delaunay defn} is dense in $\mathcal H(\kappa)$. We conclude that $\mathcal G(\kappa)$ is an open dense subset of $\mathcal H(\kappa)$.
\end{proof}

\begin{rem}
Note that condition (\ref{veering condition}) of Definition \ref{delaunay defn} is strictly weaker than the \emph{Keane condition} presented in Definition 2.9 of \cite{bell-et-al}, which requires a translation surface to have no vertical or horizontal saddle connections. The set of Keane translation surfaces is not an open subset of $\mathcal H(\kappa)$, because every translation surface has saddle connections in a dense set of directions.
\end{rem}

\begin{lem}\label{unique triangulation}
Let $M \in \mathcal G(\kappa)$. Then $M$ has a unique $L^\infty$-Delaunay triangulation.
\end{lem}

\begin{proof}
We begin by working in the universal cover $\Tilde M \twoheadrightarrow M$. For each maximal $L^\infty$-square $S: (0,h)^2 \to \Tilde M$ with 3 singularities on its boundary, the $3$ saddle connections joining these singularities do not intersect each other. Consider the collection $\Tilde \Delta$ of all saddle connections obtained in this way. Let $\Delta = p(\Tilde \Delta)$ be collection of saddle connections on $M$ that are images of saddle connections in $\Tilde \Delta$ under the universal covering map. Proposition 2.1 of \cite{gueritaud} shows that $\Delta$ is a finite triangulation of $S_g$ with vertices in $\Sigma$. By property (\ref{delaunay condition}) of Definition \ref{delaunay defn}, every triangle of $\Delta$ satisfies property (\ref{local delaunay condition}) of Definition \ref{linfty delaunay triangle defn}, and similarly for properties (\ref{veering condition}) of these definitions. Hence $\Delta$ is an $L^\infty$-Delaunay triangulation of $M$.

Let $\Delta'$ be any $L^\infty$-Delaunay triangulation of $M$. Every triangle of $\Delta$ has a circumsquare with $3$ singularities on its boundary, and hence every edge of $\Delta$ lifts along the covering projection to an edge of $\Tilde \Delta$. Therefore $\Delta' = \Delta$, and we are done.
\end{proof}

\begin{rem}\label{triangulable is generic}
It is a consequence of Lemmas \ref{open density} and \ref{unique triangulation} together with Lemma \ref{full closure region} that a translation surface $M$ has an $L^\infty$-Delaunay triangulation if and only if $M$ is $L^\infty$-generic.
\end{rem}

\begin{lem}\label{L^infty systole}
Suppose a translation surface $M$ is $L^\infty$-generic, and let $\|\cdot\|_\infty$ denote the $L^\infty$ norm on $\C = \R^2$. Then every saddle connection $\gamma$ that minimizes the quantity $\|\gamma(M)\|_\infty$ is an edge of the $L^\infty$-Delaunay triangulation $\Delta$ of $M$.
\end{lem}

The proof of Lemma \ref{L^infty systole} is a straightforward analogue of the proof of the analogous fact for $L^2$-Delaunay triangulations of translation surfaces, which can be found as Lemma 3.1 of \cite{boissy-geninska}.

\subsection{Isodelaunay regions are polytopes}

Throughout this subsection, consider $M = (X, \omega) \in \mathcal G_{\mathrm{marked}}(\kappa)$ with $L^\infty$-Delaunay triangulation $\Delta$. We endow each edge $\gamma \in E(\Delta)$ with some arbitrary orientation, so that each $\gamma \in E(\Delta)$ gives a well-defined homology class in $H_1(S_g, \Sigma; \Z)$.

To be clear about terminology, we say that a subset of $\R^d$ is a \emph{convex open polytope} if it is the intersection of finitely many open real half-spaces, i.e. if it is defined by finitely many real-linear strict inequalities. We do not require that a polytope be bounded in $\R^d$. This subsection is devoted to the proof of the following theorem.

\begin{thm}\label{polytope theorem}
Every isodelaunay region in $\mathcal H_{\mathrm{marked}}(\kappa)$ is homeomorphic to a convex open $(4g+2n-2)$-dimensional polytope via the period coordinate map $\Phi$.
\end{thm}

\subsubsection*{The polytope of complex parameters}

Let $\mathscr B = \{\gamma_1,\dots,\gamma_{2g+n-1}\}$ be the arbitrary basis of $H_1(S_g, \Sigma; \Z)$ we use to define the period coordinate map $\Phi$. For each edge $\gamma \in E(\Delta)$, we can write in homology $\gamma = \sum_j a_j \gamma_j$. For each point $\mathbf z = (z_1,\dots,z_{2g+n-1}) \in \C^{2g+n-1}$, let us write $z_\gamma = \sum_j a_j z_j$.

We shall define a polytope $\mathcal P^\circ(\Delta, v) \subset \C^{2g+n-1}$ via inequalities that the periods of any $M' \in \mathcal G_{\mathrm{marked}}(\kappa)$ must satisfy in order to have the same $L^\infty$-Delaunay triangulation as $M$. We will then show that such polytopes are the desired polytopes of Theorem \ref{polytope theorem}.

Note that a triangle whose vertices lie on the boundary of a square cannot have edges whose slopes all have the same sign. Therefore $\Delta$ is a \emph{veering triangulation}, which means precisely that no triangle of $\Delta$ has edges whose slopes all have the same sign.\footnote{In this context, this is equivalent to other definitions of veering. See Proposition 3.16 of \cite{frankel-cat}. For a discussion of veering triangulations in greater generality, see Section 2 of \cite{bell-et-al}} Let us write $\sign(x) \coloneqq |x|/x$.

\begin{defn}\label{veering conditions}
Let $v_{\Im \gamma}(M) \coloneqq \sign(\Im(\gamma(M)))$ and $v_{\Re \gamma}(M) \coloneqq \sign(\Re(\gamma(M)))$ We say that $\mathbf z \in \C^{2g+n-1}$ satisfies the \emph{veering inequalities} for $M$ if
\[
v_{\Im \gamma}(M)\Im(z_\gamma) > 0 \quad \text{and} \quad v_{\Re \gamma}(M)\Re(z_\gamma) > 0,\, \forall \gamma \in E(\Delta).
\]
We call $v(M) \coloneqq (v_{\Im \gamma}(M), v_{\Re \gamma}(M))_{\gamma \in E(\Delta)} \in \{-1,1\}^{2 |E(\Delta)|}$ the \emph{coefficient vector} for $M$.
\end{defn}

\begin{rem}\label{slope-sign equality}
By condition (\ref{veering condition}) of Definition \ref{delaunay defn} and Remark \ref{triangulable is generic}, each translation surface structure $M' \in \mathcal G_{\mathrm{marked}}(\kappa)$ assigns a slope $\Im(\gamma(M'))/\Re(\gamma(M'))$ to every $\gamma \in E(\Delta)$. If $\mathbf z = \Phi(M')$ satisfies the veering inequalities for $M$, then the numerator and denominator of the slope assigned to $\gamma$ by $M'$ have the same sign as those assigned by $M$. In particular, $\gamma$ has the same slope-sign on $M'$ as on $M$.
\end{rem}

Let $\mathrm{Quad}(M)$ denote the collection of quadrilaterals $Q$ formed by two adjacent triangles of $\Delta$ such that $M$ assigns to the sides of $Q$ alternating slope-signs. By Remark \ref{slope-sign equality}, we see that $\mathrm{Quad}(M)$ is entirely determined by $v(M)$. By Theorem 4.2 of \cite{frankel-cat}, a veering triangulation $\Delta$ of a translation surface $M$ is $L^\infty$-Delaunay if and only if, for each $Q \in \mathrm{Quad}(M)$, the edge $\gamma_Q$ that the triangles share is $L^\infty$-shorter than the other diagonal $\Gamma_Q$ of $Q$. Therefore, for each such $Q \in \mathrm{Quad}(M)$, we have the inequality $\|\gamma_Q(M)\|_\infty < \|\Gamma_Q(M)\|_\infty$.

We now show that, in the presence of the veering inequalities, for each $Q \in \mathrm{Quad}(M)$, the inequality $\|z_{\gamma_Q}\|_\infty < \|z_{\Gamma_Q}\|_\infty$ is equivalent to a pair of $\R$-linear inequalities in the real and imaginary parts of the components of $\mathbf z$. Lemma \ref{quadrilateral inequality lemma} shows this in the case where $\gamma_Q$ has positive slope, and by $90^\circ$ rotation, the analogous result holds when $\gamma_Q$ has negative slope.

\begin{lem}\label{quadrilateral inequality lemma}
Let $Q \in \mathrm{Quad}(M)$. Suppose $\mathbf z \in \C^{2g+n-1}$ satisfies the veering inequalities for $M$, and let $Q$ have negatively-sloped edges $a$ and $c$, and positively-sloped edges $b$ and $d$, with edge orientations as in Figure \ref{quad inequal picture}. Assume that ${\gamma_Q}$ has positive slope and has common endpoints with $a$ and $b$. Then
\[
\|z_{\gamma_Q}\|_\infty < \|z_{\Gamma_Q}\|_\infty \iff \Re(z_{\gamma_Q}) < \Im(z_{\Gamma_Q}) \text{\emph{ and }} \Im(z_b) < \Im(z_{\Gamma_Q}).
\]
\end{lem}

\begin{figure}
\begin{center}
\includegraphics[scale=0.5]{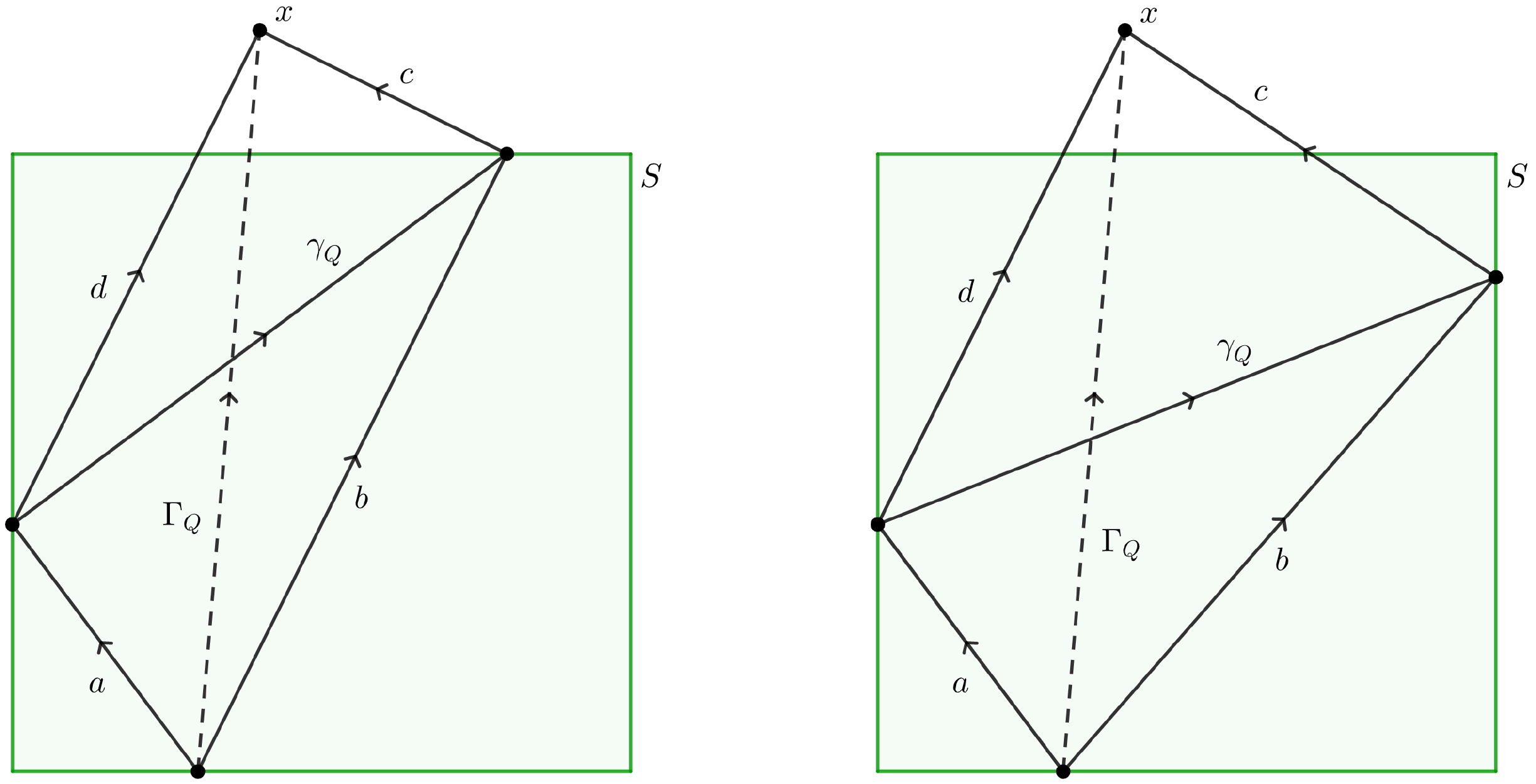}
\end{center}
\caption{Case 1 (left) and Case 2 (right) of Lemma \ref{quadrilateral inequality lemma}}
\label{quad inequal picture}
\end{figure}

\begin{proof}
For convenience, let us consider the complex numbers $z$ as vectors in the plane. Let $T$ be the triangle formed by $z_a$, $z_b$, and $z_{\gamma_Q}$, and let $S$ be the square in which it is inscribed. In Figure \ref{quad inequal picture} we abuse notation by writing e.g. $\gamma_Q$ in place of $z_{\gamma_Q}$. Let $h = \mathrm{height}(S)$, and let $x$ be the common endpoint of $z_{\Gamma_Q}$ and $z_d$. We have two cases.
\smallskip

\noindent\textbf{Case 1:} $\Re(z_{\gamma_Q}) \leq \Im(z_b)$. In this case, $\Im(z_b) = h$, and the vertices of $T$ lie on the left, bottom, and top sides of $S$, as shown in Figure \ref{quad inequal picture}. Since $d$ has positive slope and $c$ has negative slope, $x$ must lie above $S$, and hence we immediately have both $\Im(z_{\Gamma_Q}) > h = \Im(z_b) \geq \Re(z_{\gamma_Q})$ and $\|z_{\Gamma_Q}\|_\infty > h \geq \|z_{\gamma_Q}\|_\infty \geq \Re(z_{\gamma_Q})$.
\smallskip

\noindent\textbf{Case 2:} $\Re(z_{\gamma_Q}) > \Im(z_b)$. In this case, $\|z_{\gamma_Q}\|_\infty = \Re(z_{\gamma_Q}) = h$, and the vertices of $T$ lie on the left, bottom, and right sides of $S$, as shown in Figure \ref{quad inequal picture}. Suppose first that $\|z_{\Gamma_Q}\|_\infty > \|z_{\gamma_Q}\|_\infty = h$. Since $d$ has positive slope and $c$ has negative slope, $x$ must lie horizontally between the left and right sides of $S$. Therefore $\Re(z_{\Gamma_Q}) < h$ and $\Im(z_{\Gamma_Q}) > h$, so that $\|z_{\Gamma_Q}\|_\infty = \Im(z_{\Gamma_Q})$. We conclude $\Im(z_{\Gamma_Q}) = \|z_{\Gamma_Q}\|_\infty > \|z_{\gamma_Q}\|_\infty = \Re(z_{\gamma_Q}) > \Im(z_b)$.

Now suppose that $\Im(z_{\Gamma_Q}) > \Re(z_{\gamma_Q})$. Since $\|z_{\Gamma_Q}\|_\infty \geq \Im(z_{\Gamma_Q})$, we have $\|z_{\Gamma_Q}\|_\infty \geq \Im(z_{\Gamma_Q}) > \Re(z_{\gamma_Q}) = \|z_{\gamma_Q}\|_\infty$, and so we are done.
\end{proof}

\begin{defn}\label{quadrilateral conditions}
We will call the $\R$-linear inequalities given by each $Q \in \mathrm{Quad}(M)$ due to Lemma \ref{quadrilateral inequality lemma} the \emph{quadrilateral inequalities} for $M$.
\end{defn}

\begin{defn}
We have seen that $v(M)$ determines the veering and quadrilateral inequalities for $M$. Conversely, we say that a pair $(\Delta, v)$ of a triangulation $\Delta$ of $S_g$ and a vector $v \in \{-1,1\}^{2|E(\Delta)|}$ is a pair of \emph{Delaunay data} if there exists some $M \in \mathcal G_{\mathrm{marked}}(\kappa)$ such that $\Delta$ is the $L^\infty$-Delaunay triangulation of $M$ and $v = v(M)$.
\end{defn}

\begin{defn}
We define the \emph{polytope of complex parameters} for some Delaunay data $(\Delta, v)$ to be
\[
\mathcal P^\circ(\Delta, v) \coloneqq \left\{\mathbf z \in \C^{2g+n-1} \mst  \begin{array}{c} \mathbf z \text{ satisfies the veering and} \\ \text{quadrilateral inequalities given by }v \end{array} \right\}.
\]
\end{defn}

\begin{defn}[Local inverse to $\Phi$]\label{local inverse defn}
Let $\mathbf z = (z_1,\dots,z_{2g+n-1}) \in \mathcal P^\circ({\Delta, v})$. We cut $S_g$ into the triangles determined by $\Delta$, and then assign to each $\gamma \in E(\Delta)$ a side-length and angle given by the modulus and phase of $z_\gamma$. Hence we obtain a collection of Euclidean triangles, which we glue back together, endowing $(S_g,\Sigma)$ with the structure $s_{\Delta, v}(\mathbf z) \in \mathcal H_{\mathrm{marked}}(\kappa)$ of a translation surface, triangulated by saddle connections. By this construction, we see this triangulation is equal to $\Delta$ on the underlying topological surface $S_g$ of $s_{\Delta, v}(\mathbf z)$. The veering inequalities ensure that the resulting triangulation is veering, and together with the quadrilateral inequalities, they ensure that the triangulation is the $L^\infty$-Delaunay triangulation of $s_{\Delta, v}(\mathbf z)$ by Theorem 4.2 of \cite{frankel-cat} and Lemma \ref{quadrilateral inequality lemma} above. We have thus defined $s_{\Delta, v}: \mathcal P^\circ({\Delta, v}) \to \mathcal G_{\mathrm{marked}}(\kappa)$.

It is clear that $\Phi(s_{\Delta, v}(\mathbf z)) = \mathbf z$. Furthermore, when $(\Delta, v)$ are the Delaunay data for $M \in \mathcal G_{\mathrm{marked}}(\kappa)$, we have $s_{\Delta, v}(\Phi(M)) = M$. The map $s_{\Delta, v}$ is therefore a local inverse to $\Phi$ on $\mathcal P^\circ(\Delta, v)$.
\end{defn}

\begin{rem}\label{polytope in isodelaunay remark}
By the connectedness of $\mathcal P^\circ({\Delta, v})$, we have $s_{\Delta, v}(\mathcal P^\circ({\Delta, v})) \subseteq \mathcal D^\circ$ for some isodelaunay region $\mathcal D^\circ$.
\end{rem}

In the remainder of this section, we will establish Theorem \ref{polytope theorem} by showing that in fact we have equality $s_{\Delta, v}(\mathcal P^\circ({\Delta, v})) = \mathcal D^\circ$. This may be understood as saying that the veering and quadrilateral inequalities for $M$ describe the shape of the connected component of $\mathcal G_{\mathrm{marked}}(\kappa)$ containing $M$.

\subsubsection*{Delaunay limit triangulations}

Let $(\Delta, v)$ be Delaunay data. The construction of the local inverse $s_{\Delta, v}: \mathcal P^\circ({\Delta, v}) \to \mathcal G_{\mathrm{marked}}(\kappa)$ depends on gluing together Euclidean triangles built from an input vector $\mathbf z$. Such a gluing construction degenerates if one of the Euclidean triangles built from $\mathbf z$ has area 0. However, Lemma \ref{structure of T(M)} guarantees that we may still extend this construction in every such case, except when an edge of $\Delta$ has vanishing period.

\begin{defn}
Let
\[
\mathcal Z_0(\Delta) \coloneqq \{\mathbf z \in \C^{2g+n-1} \st z_\gamma=0 \text{ for some }\gamma \in E(\Delta)\}.
\]

We then define
\[
\mathcal P(\Delta, v) \coloneqq \overline{\mathcal P^\circ(\Delta, v)} \setminus \mathcal Z_0(\Delta).
\]
\end{defn}

\begin{lem}\label{structure of T(M)}
For all Delaunay data $(\Delta, v)$, the function $s_{\Delta, v}$ extends to a proper map $\mathcal P({\Delta, v}) \to \mathcal H_{\mathrm{marked}}(\kappa)$. In particular, $\mathcal P({\Delta, v})$ is the largest possible subset of $\overline{\mathcal P^\circ({\Delta, v})}$ to which the function $s_{\Delta, v}: \mathcal P^\circ({\Delta, v}) \to \mathcal G_{\mathrm{marked}}(\kappa)$ extends.
\end{lem}

\begin{proof}
We first show that $s_{\Delta, v}$ cannot extend continuously to any point of $\mathcal Z_0(\Delta)$. Every point $\mathbf z \in \overline{\mathcal P^\circ(\Delta, v)} \cap \mathcal Z_0(\Delta)$ is a limit of points $\{\Phi(M_k)\}_{k=1}^\infty$ where the $M_k$ are translation surfaces on which $\gamma$ is a saddle connection with $\|\gamma(M_k)\|_\infty < k\1$. By Masur's Compactness Criterion, it follows that the sequence $\{M_k\}_{k=1}^\infty$ has no limit in $\mathcal H_{\mathrm{marked}}(\kappa)$, and so $s_{\Delta, v}$ cannot extend to $\mathbf z$ continuously.

Now suppose $\mathbf z \in \mathcal P({\Delta, v})$. Then there is some $\e > 0$ such that for every sequence $\{M_k\}_{k=1}^\infty$ of translation surfaces with $\Phi(M_k) \to \mathbf z$, we have $\|\gamma(M_k)\|_\infty \geq \e$ for each $\gamma \in E(\Delta)$. By Lemma \ref{L^infty systole}, we have $\|\gamma(M_k)\|_\infty \geq \e$ for \emph{every} saddle connection $\gamma$ on $M_k$. By Masur's Compactness Criterion, the sequence $\{M_k\}_{k=1}^\infty$ has a limit point $M \in \mathcal H_{\mathrm{marked}}(\kappa)$.

By considering a local coordinate chart for $\Phi$ about $M$, we see that from $\Phi(M_k) \to \mathbf z$ it follows that $M$ is the unique limit point of $\{M_k\}_{k=1}^\infty$ and that $\Phi(M) = \mathbf z$. It then follows that the unique continuous extension of $s_{\Delta, v}$ to $\mathbf z$ is $s_{\Delta, v}(\mathbf z) = M$.

We conclude that $s_{\Delta, v}$ extends to all of $\mathcal P({\Delta, v})$. Since $s_{\Delta, v}$ extends to no point of $\mathcal Z_0(\Delta)$, we also conclude that this extension is proper.
\end{proof}

\begin{lem}\label{limiting circumsquare}
Each triangle $T$ of $\Delta$ has a circumsquare on every translation surface $M \in s_{\Delta, v}(\mathcal P({\Delta, v}))$.
\end{lem}

The proof is a routine exercise in taking limits. We illustrate Lemma \ref{limiting circumsquare} as follows. Let $M = s_{\Delta, v}(\mathbf z)$. If every triangle built from $\mathbf z$ has nonzero area, then it is straightforward to construct the respective circumsquare. If some triangle has area $0$, then its edges are realized by three collinear line segments. See Figure \ref{delaunay degeneration} for a sketch of how such a configuration may arise from a limit of $L^\infty$-generic surfaces. Note that the top edge of the depicted triangle tends towards a geodesic path that fails to be a saddle connection.

\begin{figure}
\begin{center}
\includegraphics[scale=0.5]{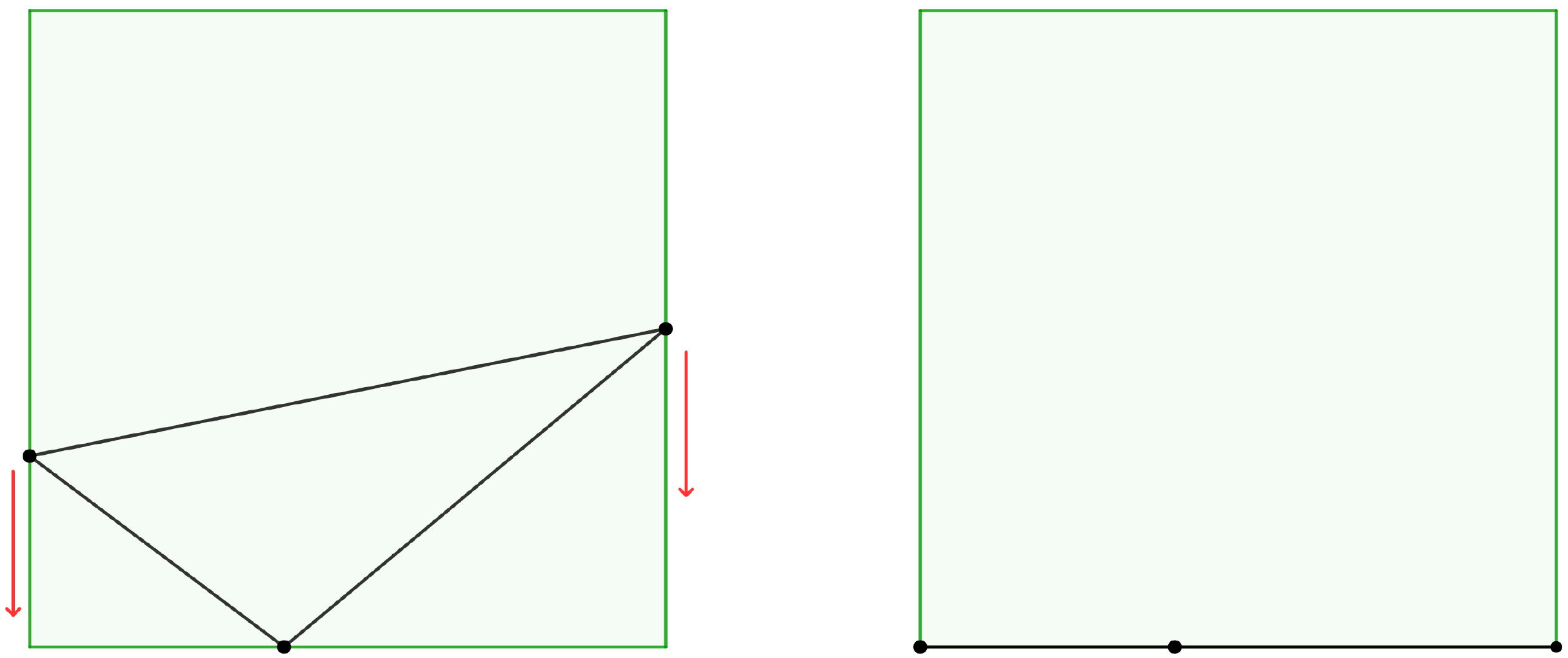}
\end{center}
\caption{$L^\infty$-Delaunay triangles tending towards a triangle of zero area}
\label{delaunay degeneration}
\end{figure}

\begin{defn}\label{delaunay limit triangulation}
Let $M \in \mathcal H_{\mathrm{marked}}(\kappa)$. Suppose $(\Delta, v)$ are Delaunay data such that there exists $\mathbf z \in \mathcal P({\Delta, v})$ with $s_{\Delta, v}(\mathbf z) = M$. Then we say that $\Delta$ is a \emph{Delaunay limit triangulation} of $M$.
\end{defn}

We are now ready to prove Theorem \ref{polytope theorem}.

\begin{proof}[Proof of Theorem \ref{polytope theorem}]
Let $\mathcal D^\circ \subset \mathcal H_{\mathrm{marked}}(\kappa)$ be an isodelaunay region. By Remark \ref{polytope in isodelaunay remark}, we have $s_{\Delta, v}(\mathcal P^\circ({\Delta, v})) \subseteq \mathcal D^\circ$, where $({\Delta, v})$ are Delaunay data for some translation surface in $\mathcal D^\circ$. We will show that $s_{\Delta,v}(\mathbf z) \notin \mathcal G_{\mathrm{marked}}(\kappa)$ for every $\mathbf z \in \mathcal P({\Delta, v}) \setminus \mathcal P^\circ(\Delta, v)$. By Lemma \ref{structure of T(M)}, it then follows that $s_{\Delta, v}\left(\mathcal P^\circ({\Delta, v})\right)$ is a maximal connected subset of $\mathcal D^\circ$. Since $\mathcal D^\circ$ is, by definition, a connected component of $\mathcal G_{\mathrm{marked}}(\kappa)$, we have $s_{\Delta, v}\left(\mathcal P^\circ({\Delta, v})\right) = \mathcal D^\circ$. Since $s_{\Delta, v}$ is the local inverse to $\Phi$, we then have $\mathcal P^\circ(\Delta, v) = \Phi(\mathcal D^\circ)$, and so we are done.

It remains to show that $M = s_{\Delta,v}(\mathbf z) \notin \mathcal G_{\mathrm{marked}}(\kappa)$ for every $\mathbf z \in \mathcal P({\Delta, v}) \setminus \mathcal P^\circ(\Delta, v)$.  If $\mathbf z$ fails one of the veering inequalities for $\gamma \in E(\Delta)$, then $\gamma$ is a vertical or horizontal geodesic path on $M$. By Lemma \ref{limiting circumsquare}, this path is inscribed in a maximal $L^\infty$-square on $M$, violating condition (\ref{veering condition}) of Definition \ref{delaunay defn}.

If $\mathbf z$ satisfies the veering inequalities and fails the quadrilateral inequalities for some quadrilateral $Q$, then by Lemma \ref{quadrilateral inequality lemma}, $Q$ must have the same height and width on $M$. Again Lemma \ref{limiting circumsquare} guarantees that $Q$ is inscribed in a maximal $L^\infty$-square on $M$, violating condition (\ref{delaunay condition}) of Definition \ref{delaunay defn}. In every case, we have $M \notin \mathcal G_{\mathrm{marked}}(\kappa)$, and so we are done.
\end{proof}

\begin{defn}
Let $(\Delta, v)$ be Delaunay data. Then we define $\mathcal D^\circ({\Delta, v}) \coloneqq s_{\Delta, v}(\mathcal P^\circ({\Delta, v}))$ and $\mathcal D(\Delta, v) \coloneqq \overline{\mathcal D^\circ(\Delta, v)}$. By Theorem \ref{polytope theorem}, every isodelaunay region of $\mathcal H_{\mathrm{marked}}(\kappa)$ is of the form $\mathcal D^\circ(\Delta, v)$.
\end{defn}

\section{The $L^\infty$-isodelaunay decomposition}\label{decomposition section}

\begin{defn}
Let $\mathscr I(\kappa)$ denote the set of all Delaunay data $(\Delta, v)$. Since $\mathcal G_{\mathrm{marked}}(\kappa)$ is dense in $\mathcal H_{\mathrm{marked}}(\kappa)$, and the isodelaunay regions are its connected components, it follows that the $L^\infty$-\emph{isodelaunay decomposition} $\mathscr D(\kappa) \coloneqq \{\mathcal D({\Delta, v})\}_{(\Delta, v) \in \mathscr I(\kappa)}$ is a covering of $\mathcal H_{\mathrm{marked}}(\kappa)$ by closed sets. We write succinctly $\mathcal D_\sigma \coloneqq \bigcap_{(\Delta, v) \in \sigma} \mathcal D(\Delta, v)$ for every $\sigma \subset \mathscr I(\kappa)$.
\end{defn}

We introduce the following structural facts about the $L^\infty$-isodelaunay decomposition, which will be useful in the proof of Lemma \ref{nerve hypotheses}.

\begin{lem}\label{full closure region}
For each $(\Delta, v) \in \mathscr I(\kappa)$, the period coordinate map $\Phi: \mathcal D(\Delta, v) \to \mathcal P(\Delta, v)$ is a homeomorphism.
\end{lem}

\begin{proof}
Since $\Phi$ is still clearly a left-inverse for $s_{\Delta, v}$ on $\mathcal P({\Delta, v})$, it follows that $s_{\Delta, v}: \mathcal P({\Delta, v}) \to \mathcal H_{\mathrm{marked}}(\kappa)$ is injective. By Theorem \ref{polytope theorem}, the image under $s_{\Delta, v}$ of the interior $\mathcal P^\circ({\Delta, v})$ of $\mathcal P({\Delta, v})$ is $\mathcal D^\circ({\Delta, v})$, and by Lemma \ref{structure of T(M)}, $s_{\Delta, v}$ is proper on $\mathcal P({\Delta, v})$. Therefore $s_{\Delta, v}$ takes $\mathcal P({\Delta, v})$ homeomorphically to $\mathcal D({\Delta, v})$. Since $s_{\Delta, v}$ is a local inverse of $\Phi$, we are done.
\end{proof}

\begin{lem}\label{convex intersections}
Every $\mathcal D_\sigma$, for $\sigma \subset \mathscr I(\kappa)$, is homeomorphic to a connected convex set via the period coordinate map $\Phi$.
\end{lem}

\begin{proof}
By Lemma \ref{full closure region}, $\mathcal D_\sigma$ is homeomorphic via $\Phi$ to
\[
\bigcap_{(\Delta, v) \in \sigma} \left(\overline{\mathcal P^\circ(\Delta, v)} \setminus \mathcal Z_0(\Delta)\right) = \bigcap_{(\Delta, v) \in \sigma} \overline{\mathcal P^\circ(\Delta, v)} \setminus \bigcup_{(\Delta, v) \in \sigma} \mathcal Z_0(\Delta).
\]
The set $\bigcap_{(\Delta, v) \in \sigma} \overline{\mathcal P^\circ(\Delta, v)}$ is connected and convex since each $\overline{\mathcal P^\circ(\Delta, v)}$ is. It remains to show that subtracting the locus $\bigcup_{(\Delta, v) \in \sigma} \mathcal Z_0(\Delta)$ preserves connectedness and convexity. This locus is the union of all $\{z_\gamma = 0\}$ where $\gamma \in E(\Delta)$ for some $(\Delta, v) \in \sigma$. The nonstrict veering inequalities for the $(\Delta, v)$ together define a closed convex polytope in which $\bigcap_{(\Delta, v) \in \sigma} \overline{\mathcal P^\circ(\Delta, v)}$ lies. Each set $\{z_\gamma = 0\}$ is then the intersection of two bounding hyperplanes of this polytope, and hence $\bigcap_{(\Delta, v) \in \sigma} \overline{\mathcal P^\circ(\Delta, v)} \setminus \{z_\gamma = 0\}$ remains connected and convex. Continuing for each $\gamma$, we conclude that $\mathcal D_\sigma$ is connected and convex.
\end{proof}

\setcounter{subsection}{1}
\subsection{The nerve of the decomposition}

Let us recall the following definitions.

\begin{defn}[Nerve of a covering]
Let $\mathscr A = \{A_i\}_{i \in I}$ be a covering of a topological space $X$ by subsets $A_i$, and let $A_\sigma \coloneqq \bigcap_{i \in \sigma} A_i$ for every $\sigma \subset I$. The \emph{nerve} of $\mathscr A$ is the simplicial complex $\mathcal N(\mathscr A)$ that has a vertex $v_i$ for every $i \in I$, and for every $\sigma \subset I$, the vertices $v_i$ for $i \in \sigma$ span a simplex if and only if $A_\sigma \ne \emptyset$.

We say that $\mathscr A$ is \emph{locally finite} if for every $x \in X$, there exists an open neighborhood $U \ni x$ such that $U \cap A_i \ne \emptyset$ for only finitely many $i \in I$. Note that local finiteness implies that $A_\sigma$ is empty for every infinite subset $\sigma \subset I$.
\end{defn}

\begin{defn}[Invariant covering]
Let $\mathscr A = \{A_i\}_{i \in I}$ be a covering of $X$. If a group $G$ acts on $X$, we say that $\mathscr A$ is $G$\emph{-invariant} if the $G$-action induces an action on $\mathscr A$; that is, for every $g \in G$ and $i \in I$, there is a $j \in I$ with $g(A_i) = A_j$. We denote by $G_\sigma$ the \emph{stabilizer subgroup} of $G$ whose elements map $A_\sigma$ onto itself. Observe that if $\mathscr A$ is $G$-invariant, then $G$ acts on $\mathcal N(\mathscr A)$ by simplicial automorphisms.
\end{defn}

\begin{defn}
Let $G$ be a group acting on spaces $X$ and $Y$. We say that a homotopy $H:X \times [0,1] \to Y$ is a $G$\emph{-homotopy} if $H(gx, t) = gH(x, t)$ for every $x \in X$, $g \in G$, and $t \in [0,1]$. We say $A \subset X$ is a $G$\emph{-deformation retract} of $X$ if there is a $G$-homotopy $H: X \times [0,1] \to X$ with $H(x,0) = x$, $H(x,1) \in A$, and $H(a,t) = a$ for every $x \in X$, $a \in A$, and $t \in [0,1]$.

We say that $X$ and $Y$ are $G$\emph{-homotopy equivalent} if there are maps $f:X \to Y$ and $h:Y \to X$ such that $h \circ f$ and $f \circ h$ are $G$-homotopic to the identity maps on $X$ and $Y$, respectively. When $Y$ is a point, we say $X$ is $G$\emph{-contractible}.
\end{defn}

We will also make use of the following nonstandard but straightforward definition.

\begin{defn}
Let $\mathscr A = \{A_i\}_{i \in I}$ and $\mathscr B = \{B_i\}_{i \in I}$ be coverings of $X$ with the same index set $I$. We say that $\mathscr A$ is an \emph{equivalent refinement} of $\mathscr B$ if $A_\sigma$ is a $G_\sigma$-deformation retract of $B_\sigma$ for every $\sigma \subset I$, and the map $\mathcal N(\mathscr A) \to \mathcal N(\mathscr B)$ given by $A_\sigma \mapsto B_\sigma$ is an isomorphism.
\end{defn}

In this section we will prove the following proposition.

\begin{prop}\label{isodelaunay nerve lemma}
There is an $\mathrm{MCG}(S_g, \Sigma)$-homotopy equivalence
\[
\mathcal H_{\mathrm{marked}}(\kappa) \simeq_{\mathrm{MCG}(S_g, \Sigma)} \mathcal N(\mathscr D(\kappa)).
\]
\end{prop}

There exists in the literature a variety of equivariant versions of the Nerve Lemma, which appears in its most basic form as Corollary 4G.3 of \cite{hatcher}. We will derive Lemma \ref{isodelaunay nerve lemma} as a corollary to our equivariant version, Theorem \ref{equivariant nerve lemma}, a generalization of Theorem 4.6 of \cite{gonzalez-gonzalez}, which is itself a modification of Lemma 2.5 of \cite{hess-hirsch}. See also Proposition 2.2 of \cite{paris} for the case of a group acting freely.

\begin{thm}\label{equivariant nerve lemma}
Let $G$ be a discrete group acting properly discontinuously on a paracompact Hausdorff space $X$. Let $\mathscr A = \{A_i\}_{i \in I}$ be a locally finite $G$-invariant closed covering of $X$ that is an equivalent refinement of a $G$-invariant open covering $\mathscr U = \{U_i\}_{i \in I}$. Suppose that each $G_\sigma$ is finite and each $A_\sigma$ is $G_\sigma$-contractible. Then there is a $G$-homotopy equivalence
\[
X \simeq_G \mathcal N(\mathscr A).
\]
\end{thm}

We will prove Theorem \ref{equivariant nerve lemma} in Appendix A.

\begin{lem}\label{nerve hypotheses}
The $L^\infty$-isodelaunay decomposition $\mathscr D(\kappa)$ satisfies the hypotheses of Theorem \ref{equivariant nerve lemma}:
\begin{enumerate}
\item\label{good cover} $\mathscr D(\kappa)$ is locally finite and $\mathrm{MCG}(S_g, \Sigma)$-invariant,
\item\label{finite isotropy} The stabilizer subgroup $\mathrm{MCG}(S_g, \Sigma)_\sigma$ is finite for every $\sigma \subset \mathscr I(\kappa)$,
\item\label{equivariant contractibility} The set $\mathcal D_\sigma$ is $\mathrm{MCG}(S_g, \Sigma)_\sigma$-contractible for every $\sigma \subset \mathscr I(\kappa)$,
\item\label{equivalent refinement check} $\mathscr D(\kappa)$ is an equivalent refinement of a $G$-invariant open covering.
\end{enumerate}
\end{lem}

\begin{proof}
Throughout, let $G = \mathrm{MCG}(S_g, \Sigma)$. Local finiteness of $\mathscr D(\kappa)$ is proved in Corollary A.5 of \cite{frankel-comparison}. Now, recall that $\mathcal G_{\mathrm{marked}}(\kappa)$ is the preimage of $\mathcal G(\kappa)$ under the quotient map $\mathcal H_{\mathrm{marked}}(\kappa) \twoheadrightarrow \mathcal H(\kappa)$ by $G$, and so the action of $G$ restricts to an action on $\mathcal G_{\mathrm{marked}}(\kappa)$. This induces an action on $\pi_0(\mathcal G_{\mathrm{marked}}(\kappa)) = \{\mathcal D^\circ(\Delta, v)\}_{(\Delta, v) \in \mathscr I(\kappa)}$, and hence also on $\mathscr D(\kappa) = \{\mathcal D(\Delta, v)\}_{(\Delta, v) \in \mathscr I(\kappa)}$. This establishes (\ref{good cover}).

Let $\sigma \subset \mathscr I(\kappa)$. If $g \in G$ satisfies $g(\mathcal D_\sigma) = \mathcal D_\sigma$, then $g$ is the mapping class $[\phi]$ of a homeomorphism $\phi: S_g \to S_g$ such that for each $(\Delta, v) \in \sigma$, the triangulation $\phi(\Delta)$ is isotopic rel $\Sigma$ to $\Delta'$ for some $(\Delta', v') \in \sigma$. Since $\sigma$ is finite, and there exist only finitely many mapping classes $[\phi]$ such that $\phi(\Delta)$ is isotopic rel $\Sigma$ to $\Delta$, we conclude that $G_\sigma$ is finite. This establishes (\ref{finite isotropy}).

By Lemma \ref{convex intersections}, each $\mathcal D_\sigma$ is connected and convex. Recall from Remark \ref{induced linear structure} that $G$ acts on $\mathcal H_{\mathrm{marked}}(\kappa)$ by affine diffeomorphisms. Therefore each stabilizer subgroup $G_\sigma$ acts on $\mathcal D_\sigma$ by linear automorphisms. Now, let $M' \in \mathcal D_\sigma$ be arbitrary. By convexity of $\mathcal D_\sigma$ and linearity of the group action, there is a point
\[
M_{\text{avg}} \coloneqq \frac{1}{|G_\sigma|} \sum_{g \in G_\sigma} g(M') \in \mathcal D_\sigma
\]
fixed by $G_\sigma$. Also by convexity, the set $\mathcal D_\sigma$ is star-shaped with respect to $M_{\text{avg}}$, and hence we have a straight-line homotopy
\begin{align*}
\mathcal D_\sigma \times [0,1] &\to \mathcal D_\sigma \\
(M, t) &\mapsto (1-t)M + tM_{\text{avg}}.
\end{align*}
This homotopy is a deformation retraction of $\mathcal D_\sigma$ onto the point $\{M_{\text{avg}}\}$, and is clearly $G_\sigma$-equivariant. Therefore we have established (\ref{equivariant contractibility}). We defer the proof of (\ref{equivalent refinement check}) to Lemma \ref{D(k) is a refinement}.
\end{proof}

\begin{proof}[Proof of Proposition \ref{isodelaunay nerve lemma}]
This is now an immediate consequence of Theorem \ref{equivariant nerve lemma} and Lemma \ref{nerve hypotheses}.
\end{proof}

\section{The infinite adjacency phenomenon}\label{infinite adjacency section}

\begin{prop}[The infinite adjacency phenomenon]\label{infinite adjacency phenomenon}
The isodelaunay nerve $\mathcal N(\mathscr D(\kappa))$ is not locally finite.
\end{prop}

Recall that a simplicial complex is \emph{locally finite} if every vertex belongs to only finitely many simplices. Failure of local finiteness for $\mathcal N(\mathscr D(\kappa))$ means that there exist $\mathcal D(\Delta, v)$ such that for infinitely many other $\mathcal D(\Delta', v')$, we have $\mathcal D(\Delta, v) \cap \mathcal D(\Delta', v') \ne \emptyset$.

\begin{rem}\label{proper subfaces}
As a corollary to Proposition \ref{infinite adjacency phenomenon}, observe that, while $\mathcal D(\Delta, v)$ and $\mathcal D(\Delta', v')$ are polytopes (minus some intersections of bounding hyperplanes, see Lemma \ref{full closure region}), the intersection $\mathcal D(\Delta, v) \cap \mathcal D(\Delta', v')$ is not always a face of $\mathcal D(\Delta, v)$. While the intersection is of course a lower-dimensional polytope, it is sometimes only a proper subset of a face of $\mathcal D(\Delta, v)$. If the intersection were always a face of $\mathcal D(\Delta, v)$, then $\mathcal N(\mathscr D(\kappa))$ would be locally finite, since each $\mathcal D(\Delta, v)$ has finitely many faces, and the covering $\mathscr D(\kappa)$ is locally finite.
\end{rem}

\begin{rem}\label{frankel's remark}
In the discussion at the end of Section 4 of \cite{frankel-cat}, Frankel makes the following remark (emphasis added).

\begin{quote}
It is possible to write a finite orbifold cover of $QD(\mathcal M_{g,n})$ as a finite union of convex sets with disjoint interiors, based on the topological type of the $L^\infty$ Delaunay triangulation. However, already in the case $\mathcal M_{1,1}$, which has lowest complexity of all moduli spaces one might try to consider, \emph{a cell of top dimension may have infinitely many cells of lower dimension on its boundary}. (The author is grateful to Simion Filip for pointing out this complication.) However, one might expect that the infinite families can be parametrized by finite data, for instance; one may hope that all neighbors of a fixed cell are obtained from one of finitely many other cells up to Dehn twists.
\end{quote}
The $QD(\mathcal M_{g,n})$ that Frankel considers is the moduli space of quadratic differentials, analogous to our $\mathcal H(\kappa)$ for abelian differentials, and the convex sets are the closures of isodelaunay regions. The emphasized text refers to the infinite adjacency phenomenon in the case of $\mathcal H(0)$, where the cells of lower dimension are the lower-dimensional polytopes discussed in Remark \ref{proper subfaces}. The anticipated finiteness result is realized by our Theorem \ref{classification result}, which allows us to identify an infinite collection $\mathcal T(\kappa)$ of adjacencies between isodelaunay regions whose deletion gives rise to the finite simplicial complex $\mathcal I(\kappa)$ of Theorem \ref{intro main}.
\end{rem}

\begin{defn}\label{cylinder definition}
A \emph{cylinder} in a translation surface $M = (X, \omega)$ is an embedding
\[
C:\R/w\Z \times (0,h) \hookrightarrow X, \quad \quad w,h >0
\]
satisfying $C^*\omega = e^{i\theta}dz$ that is maximal in the sense that the closure of its image has singularities of $M$ on each boundary component. Since a cylinder is an embedding, we will often identify it with its image. We call the image of each $\R/w\Z \times \{s\}$ a \emph{core curve} of $C$, for $0 < s < h$. Every closed Euclidean geodesic on $M$ is the core curve of some cylinder.

Let us assume $0 \leq \theta < \pi$. We write
\[
\mathrm{width}(C) = w,\quad \mathrm{height}(C) = h,\quad \mathrm{mod}(C) = \tfrac hw,\quad \mathrm{arg}(C) = \theta.
\]
When $\mathrm{arg}(C) = 0$ we say $C$ is \emph{horizontal}, and when $\mathrm{arg}(C) = \tfrac\pi2$ we say $C$ is \emph{vertical}.
\end{defn}

The following lemma is an elementary computation, whose proof we omit. It will be used in Lemmas \ref{tilted triangulations are optimal} and \ref{modulus bound}.

\begin{lem}\label{displacement formula}
Let $C$ be a horizontal cylinder in a translation surface $M$, and let $\gamma$ and $\gamma'$ be saddle connections lying in $C$ such that $\#(\gamma \cap \gamma') \geq 2$. Let $\Tilde{\gamma}$ be a lift of $\gamma$ to the universal cover $\Tilde{M} \twoheadrightarrow M$. Let $z_1, z_2, \dots, z_k$ be the preimages on $\Tilde{\gamma}$ of the points of $\gamma \cap \gamma'$, ordered consecutively along $\Tilde{\gamma}$. Then the horizontal and vertical distances from each $z_i$ to $z_{i+1}$ are, respectively,
\[
\left|\frac{\slope(\gamma')\width(C)}{\slope(\gamma) - \slope(\gamma')}\right| \quad \quad \text{and} \quad \quad \left|\frac{\slope(\gamma)\slope(\gamma')\width(C)}{\slope(\gamma) - \slope(\gamma')}\right|
\]
\end{lem}

The following lemma is a basic computation of edges of Delaunay limit triangulations, which will be used in Lemmas \ref{adjacency computation} and \ref{non-decreasing intersection numbers}. Recall from Remark \ref{gl2r action} that $\GL_2^+(\R)$ acts on $\mathcal H_{\mathrm{marked}}(\kappa)$. In the following lemma, we establish the notation $(\Delta_M^\pm, v_M^\pm)$.

\begin{lem}\label{edge computation} Let $C$ be a horizontal cylinder in a translation surface $M$ with $\mathrm{mod}(C)>1$.
\begin{enumerate}
\item\label{tilted triangulation} Let $A_\e = \left( \begin{smallmatrix} 1 & 0 \\ -\e & 1+\e \end{smallmatrix} \right) \in \GL_2^+(\R)$. Then there exist $(\Delta_M^+, v_M^+), (\Delta_M^-, v_M^-) \in \mathscr I(\kappa)$ such that $A_\e M \in \mathcal D(\Delta_M^+, v_M^+)$ for every $0 < \e \ll 1$, and $A_\e M \in \mathcal D(\Delta_M^-, v_M^-)$ for every $-1 \ll \e < 0$. Hence $M \in \mathcal D(\Delta_M^+, v_M^+) \cap \mathcal D(\Delta_M^-, v_M^-)$.

\item\label{diagonal saddle} Let $S$ be a maximal $L^\infty$-square in $C$, one of whose vertices is a singularity $x$ of $M$. Then the saddle connection $\gamma$ that minimizes $|\mathrm{slope}(\gamma)|$ among saddle connections inscribed in $S$ and starting at $x$ is an edge of $\Delta_M^+$  if $\mathrm{slope}(\gamma) > 0$, and is an edge of $\Delta_M^-$ if $\mathrm{slope}(\gamma) < 0$.
\end{enumerate}
By a $90^\circ$ rotation, the analogous result follows when $C$ is vertical.
\end{lem}

\begin{rem}
The matrix $A_\e$ is chosen because it sends horizontal lines to lines of nonzero slope, it preserves the horizontal coordinate of every vector, and $\e \mapsto A_\e M$ is linear in period coordinates. These latter two properties are for convenience; a rotation matrix would also suffice.
\end{rem}

\begin{proof}[Proof of Lemma \ref{edge computation}]
Because the path $\e \mapsto A_\e M$ is linear in period coordinates, and $\mathscr D(\kappa)$ is a locally finite covering by polytopes, there exists some $\mathcal D(\Delta_M^+, v_M^+)$ in which $A_\e M$ lies for every $0 < \e \ll 1$. Similarly for the path parametrized by negative $\e$. This establishes (\ref{tilted triangulation}).

Let $w = \mathrm{width}(C)$ and $h = \mathrm{height}(C)$. Assume first that $x$ is the bottom left vertex of $S$, so that $x = \overline{S}(0,0)$. Let $y$ be the first singularity on $M$ to the right of $x$ (it may happen that $y=x$), so $y = \overline{S}(\ell,0)$ for some $0 < \ell \leq w$. The singularities on the top of $S$ are of the form $\overline{S}(h-\delta,h)$. Let $z$ be the singularity that minimizes $\delta$. Note that $0 \leq \delta < w$. Then the image under $\overline{S}$ of the line segment from $(0,0)$ to $(h-\delta,h)$ is our slope-minimizing saddle connection $\gamma$.

Let us see why $x$, $y$, and $z$ form a Delaunay limit triangle. Let $0 < \e \ll 1$. Then $A_\e M$ has an $L^\infty$-square of height $h+\e(\delta + \ell)$ with only $x$, $y$, and $z$ on its boundary, as shown in Figure \ref{deformed triangle}. Thus the singularities $x$, $y$, and $z$ form an $L^\infty$-Delaunay triangle $T$ on $A_\e M$. Since a triangle's being $L^\infty$-Delaunay is an open condition, $T$ must be a triangle of $\Delta_M^+$, even if $A_\e M$ is not  $L^\infty$-generic.

If $x$ is instead the top right vertex of $S$, then we again apply $A_\e$ to obtain Figure \ref{deformed triangle} rotated by $180^\circ$, and hence $T$ is a triangle of $\Delta_M^+$. If $x$ is the bottom right or top left vertex of $S$ (so $\mathrm{slope}(\gamma)<0$), then we instead apply $A_\e$ for $-1 \ll \e < 0$ in order to obtain a suitably rotated and reflected version of Figure \ref{deformed triangle}, and hence $T$ is a triangle of $\Delta_M^-$ in these cases.
\begin{figure}
\begin{center}
\includegraphics[scale=0.6]{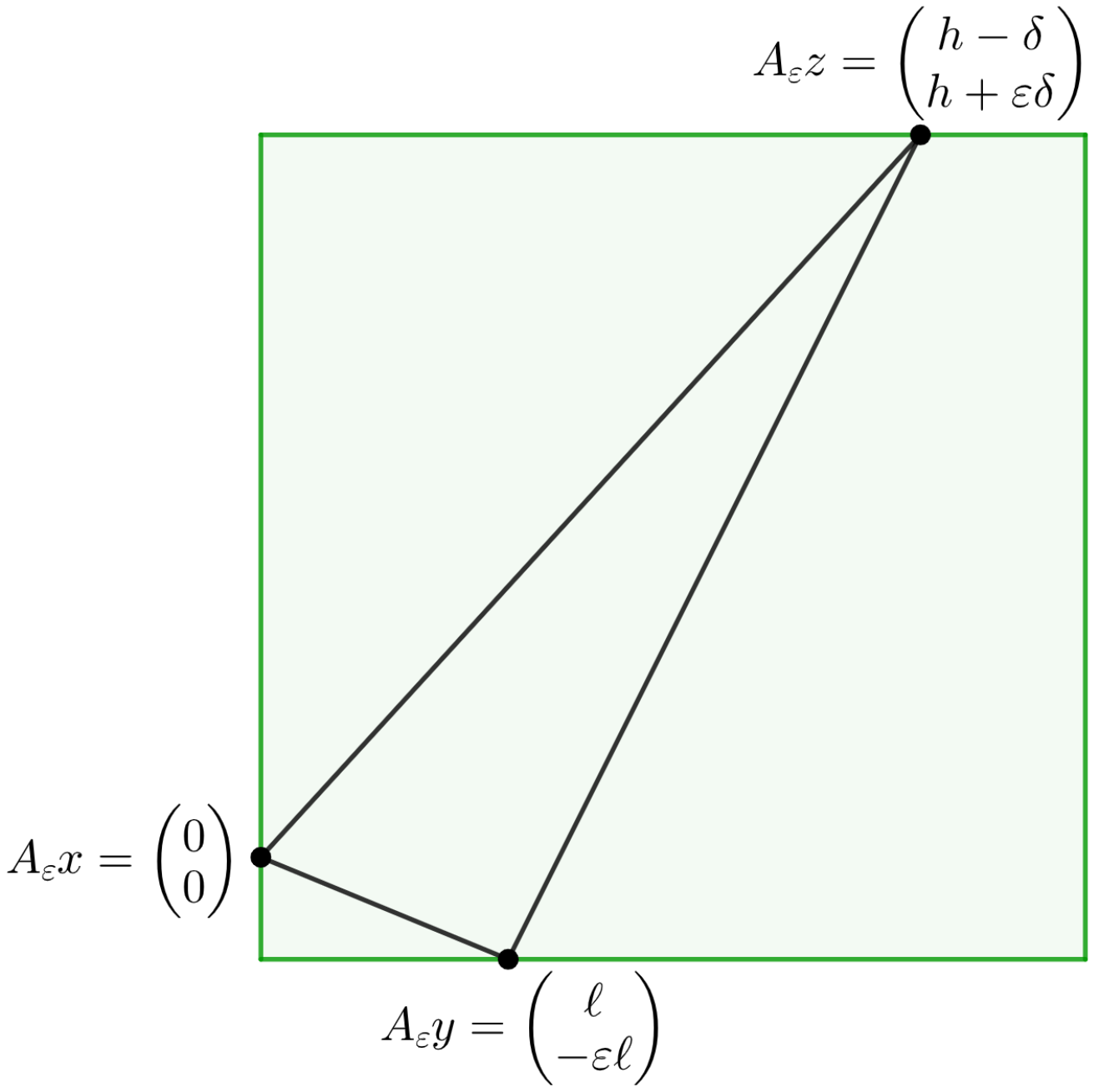}
\end{center}
\caption{As $\e \to 0$, we obtain a Delaunay limit triangle in $C$}
\label{deformed triangle}
\end{figure}
\end{proof}

\begin{defn}
The \emph{intersection number} of two arcs $\gamma, \gamma'$ on $S_g$ is
\[
\iota(\gamma, \gamma') \coloneqq \min \#(p \cap p')
\]
where $p$ and $p'$ range over all arcs isotopic to $\gamma$ and $\gamma'$ rel $\Sigma$. We consider arcs to be parametrized by open intervals, so that intersections at the endpoints are not counted.\end{defn}

Note that when the arcs $\gamma$ and $\gamma'$ are saddle connections on a translation surface that meet transversely, we have
\[
\iota(\gamma, \gamma') = \# (\gamma \cap \gamma').
\]

\begin{defn}
Consider $M \in \mathcal H_{\mathrm{marked}}(\kappa)$, and let $\sigma \subset \mathscr I(\kappa)$ be the set of Delaunay limit triangulations of $M$, hence $M \in \mathcal D_\sigma$. We define
\[
\iota(\sigma) = \iota(M) \coloneqq \max \{\iota(\gamma, \gamma') \st \gamma \in E(\Delta), \gamma' \in E(\Delta') \text{ for some } (\Delta,v),(\Delta',v') \in \sigma\}.
\]
Furthermore, for a cylinder $C$ in $M$, define $\iota(M; C)$ analogously, where $\gamma$ and $\gamma'$ range only over those edges lying in $C$.
\end{defn}

\begin{rem}\label{isotopy class remark}
Let $(\Delta, v)$ be a Delaunay limit triangulation of a translation surface $M$. Recall from Figure \ref{delaunay degeneration} that the geodesic representative on $M$ of an edge $\gamma$ of $\Delta$ may contain points of $\Sigma$ on its interior when this representitive is horizontal or vertical. Therefore, this geodesic representative may fail to to be isotopic rel $\Sigma$ to the arc $\gamma$, although it is of course a limit of arcs in the relative isotopy class of $\gamma$. It is therefore sometimes required to pass to non-geodesic representatives of arcs $\gamma \in E(\Delta)$, $\gamma' \in E(\Delta')$ in order to compute $\iota(\gamma, \gamma')$. We treat this situation in Lemma \ref{controlling degenerate edges}.
\end{rem}

It follows from Lemma \ref{cylinder without hypothesis} that if $\iota(M) \geq 3$, then $\iota(M) = \iota(M; C)$, where $C$ is some horizontal or vertical cylinder of modulus greater than $1$.

\begin{lem}\label{tilted triangulations are optimal}
Let $C$ a horizontal cylinder of modulus greater than 1 in a translation surface $M$. Then
\[
\iota(M; C) = \max \{\iota(\gamma, \gamma') \st \gamma \in E(\Delta_M^+), \gamma' \in E(\Delta_M^-) \text{ lying in }C\},
\]
and furthermore, this maximum is realized by saddle connections $\gamma$, $\gamma'$ of the form (\ref{diagonal saddle}) in Lemma \ref{edge computation} and satisfying $\slope(\gamma') < 0 < \slope(\gamma)$. By a $90^\circ$ rotation, the analogous result follows when $C$ is vertical.
\end{lem}

\begin{proof}
Let $\gamma \in E(\Delta)$ and $\gamma' \in E(\Delta')$ lie in a cylinder $C$ such that $\iota(M; C) = \iota(\gamma, \gamma')$. Up to reflection about a vertical line, we may suppose at least one of $\gamma$, $\gamma'$ has positive slope; let $\gamma$ have positive slope, and if both have positive slope, then let $\slope(\gamma) < \slope(\gamma')$. Let $S$ be a maximal $L^\infty$-square in $C$ whose bottom left vertex is an endpoint $x$ of $\gamma$. Let $\gamma^+ \in E(\Delta_M^+)$ be the saddle connection that minimizes $|\slope(\gamma^+)|$ among saddle connections inscribed in $S$ and starting at $x$. If $\gamma \ne \gamma^+$, then $\gamma^+$ intersects $\gamma'$ nearer to the bottom endpoint of $\gamma'$ than does $\gamma$, and Lemma \ref{displacement formula} shows that the vertical distance between a consecutive pair of points in $\gamma^+ \cap \gamma'$ is less than that in $\gamma \cap \gamma'$. Hence $\iota(\gamma^+, \gamma') \geq \iota(\gamma, \gamma')$.

Let $R$ be the maximal $L^\infty$-square in $C$ whose bottom right vertex is an endpoint $y$ of $\gamma'$. Let $\gamma^- \in E(\Delta_M^-)$ be the saddle connection that minimizes $|\slope(\gamma^-)|$ among saddle connections inscribed in $S$ and starting at $y$. If $\gamma' \ne \gamma^-$, then $\gamma^-$ intersects $\gamma^+$ nearer to the bottom endpoint of $\gamma^+$ than does $\gamma'$, and Lemma \ref{displacement formula} shows that the vertical distance between a consecutive pair of points in $\gamma^+ \cap \gamma^-$ is less than that in $\gamma^+ \cap \gamma'$. Hence $\iota(\gamma^+, \gamma^-) \geq \iota(\gamma^+, \gamma')$. Since $\gamma^+ \in E(\Delta_M^+)$ and $\gamma^- \in E(\Delta_M^+)$ are of the form (\ref{diagonal saddle}) in Lemma \ref{edge computation} and satisfy $\slope(\gamma^-) < 0 < \slope(\gamma^+)$, we are done.
\end{proof}

\begin{defn}\label{stretch definition}
Let $C$ be a cylinder with argument $\theta$ in a translation surface $M$, and let $R_\theta \coloneqq \left(\begin{smallmatrix} \cos\theta & -\sin\theta \\ \sin\theta & \cos\theta \end{smallmatrix}\right)$. For $t > 0$, we define the \emph{cylinder stretch} $a_t^C$, and for $t \in \R$ we define \emph{cylinder shear} $u_t^C$ as follows.

When $t > 0$, we may obtain a new translation surface $a_t^C(M)$ by acting on $C$ via the matrix $R_\theta\left(\begin{smallmatrix} 1 & 0 \\ 0 & t \end{smallmatrix}\right)R_{-\theta}$ while leaving the rest of $M$ unchanged. For any $t\in\R$, we may similarly act on $C$ via the matrix $R_\theta\left(\begin{smallmatrix} 1 & t \\ 0 & 1 \end{smallmatrix}\right)R_{-\theta}$ while leaving the rest of $M$ unchanged to obtain $u_t^C(M)$.

Observe that $a_t^C(M)$ and $u_t^C(M)$ have a cylinder where $C$ used to be, and so by abuse of notation we may let $C$ denote this cylinder. Note that $u_s^C \circ a_t^C = a_t^C \circ u_{st}^C$.
\end{defn}

The following lemma produces a path $M_t$ along the boundary of an isodelaunay region that passes through the boundaries of infinitely many other isodelaunay regions; see the subsequent proof of Proposition \ref{infinite adjacency phenomenon}.

\begin{lem}\label{adjacency computation}
Let $M \in \mathcal H_{\mathrm{marked}}(\kappa)$ have a horizontal cylinder C, and let
\[
M_t \coloneqq a_t^C \circ u_{\pm(t-1)}^C(M).
\]
Then for all $s,t > \tfrac{2}{\mathrm{mod}(C)}$, we have $\mathcal D\left(\Delta_{M_s}^\pm, v_{M_s}^\pm\right) = \mathcal D\left(\Delta_{M_t}^\pm, v_{M_t}^\pm\right)$. Furthermore, we have $\lim_{t \to \infty} \iota(M_t) = \infty$. By a $90^\circ$ rotation, the analogous result follows when $C$ is vertical.
\end{lem}

\begin{proof}
Fix some $t > 2/\modulus(C)$, and let $w = \width(C)$ and $h = \hei(C)$. On the surface $M_t = a_t^C \circ u_{t-1}^C(M)$, the horizontal cylinder $C$ has width $w$ and height
\[
th > 2h/\modulus(C) = 2w.
\]
Let $x$ be a singularity of $M_t$ on the boundary of $C$. Let us consider an $L^\infty$-square $S:(0,th)^2 \to M_t$ whose image is contained in $C$, such that $x = \overline{S}(0,0)$. By Lemma \ref{edge computation}, the saddle connection $\gamma$ that minimizes $|\mathrm{slope}(\gamma)|$ among saddle connections inscribed in $S$ and starting at $x$ belongs to a triangle $T_x$ of the triangulation $\Delta_{M_t}^+$ defined in the same lemma.

As in Lemma \ref{edge computation}, let $z = \overline{S}(th-\delta,th)$ be the other endpoint of $\gamma$, and let $y = \overline{S}(\ell,0)$ be the first singularity to the right of $x$, where $0 < \ell \leq w$ and $0 \leq \delta < w$. Since $\ell \leq w < th-\delta$, the slopes of both of the non-horizontal edges of $T_x$ are positive. After possibly further restricting how small $\e$ must be in part (\ref{tilted triangulation}) of Lemma \ref{edge computation}, we obtain Delaunay limit triangles $T_{x'}$ on $M_t$ for each singularity $x'$ on the bottom of $C$. Notice that the other $L^\infty$-Delaunay triangles lying in $C$ (those with two sides on the top of $C$) are entirely determined by these triangles $T_{x'}$, and that similarly, all their non-horizontal edges have positive slope.

We claim that $M_s \in \mathcal D(\Delta_{M_t}^+, v_{M_t}^+)$ for all $s > t$. Observe that $a_s^C \circ u_{s-1}^C$ acts on $C$ via the matrix $B_s = \left(\begin{smallmatrix} 1 & s-1 \\ 0 & s \end{smallmatrix}\right)$, and let $A_\e = \left( \begin{smallmatrix} 1 & 0 \\ -\e & 1+\e \end{smallmatrix} \right)$ as in Lemma \ref{edge computation}. Since $(\ell,0) = B_s(\ell,0)$ and $(sh-\delta, sh) = B_s(h-\delta, h)$, the triangle $T_x$ on $A_\e M_s$ is of the form in Figure \ref{deformed triangle}, with $sh$ in place of $h$. Therefore each $T_x$ is an $L^\infty$-Delaunay triangle of $A_\e M_s$, with edges having the same slope-sign as on $A_\e M_t$, for every $s > t$ and every $0 < \e \ll 1$. Since $a_s^C \circ u_{s-1}^C$ leaves the complement of $C$ in $M$ unchanged, every triangle $T$ of $\Delta_{M_t}^+$ not lying in $C$ is also an $L^\infty$-Delaunay triangle of $A_\e M_s$ for every $s > t$ and every $0 < \e \ll 1$. We conclude that $M_s \in \mathcal D(\Delta_{M_t}^+, v_{M_t}^+)$ for all $s > t$.

It remains to show that $\lim_{t \to \infty} \iota(M_t) = \infty$. Consider again some fixed singularity $x$ on the bottom of $C$, the maximal $L^\infty$ square $S$, and the $\gamma$ of minimal slope with endpoints $x = \overline{S}(0,0)$ and $z = \overline{S}(th-\delta,th)$. Then
\[
\lim_{t \to \infty} \mathrm{slope}(\gamma) = \lim_{t \to \infty} \frac{th-\delta}{th} = 1.
\]

Let $R:(0,th)^2 \to M_t$ be an $L^\infty$-square whose image is contained in $C$, such that $x = \overline{R}(th,0)$. Let $\gamma^-_t \in E(\Delta_{M_t}^-)$ be the saddle connection given by part (\ref{diagonal saddle}) of Lemma \ref{edge computation} for this square $R$. Similarly to $\gamma$, the endpoints of $\gamma^-_t$ are $x = \overline{R}(th,0)$ and $\overline{R}(\delta'_t, th)$, where $0 \leq \delta'_t < w$. Thus $\lim_{t \to \infty} \mathrm{slope}(\gamma'_t) = -1$. The height $th$ of $C$ tends to $\infty$ as $t \to \infty$, while the width remains $w$. Therefore
\[
\lim_{t \to \infty}\iota(\gamma, \gamma^-_t) = \infty.
\]
We conclude that
\[
\lim_{t \to \infty} \iota\left( M_t \right) = \infty.
\]
We have shown the desired result for $M_t = a_t^C \circ u_{t-1}^C(M)$, and reversing the roles of $\Delta^+_M$ and $\Delta^-_M$ gives the desired result for $M_t = a_t^C \circ u_{-(t-1)}^C(M)$, so we are done.
\end{proof}

\begin{proof}[Proof of Proposition \ref{infinite adjacency phenomenon}]
Let $\mathcal D(\Delta, v) = \mathcal D(\Delta_{M_t}^+, v_{M_t}^+)$ as in Lemma \ref{adjacency computation}. If there were only finitely many other  $\mathcal D(\Delta', v')$ such that $\mathcal D(\Delta, v) \cap \mathcal D(\Delta', v') \ne \emptyset$, then we would have $\sup_{\sigma \ni (\Delta, v)} \iota(\sigma) < \infty$. But Lemma \ref{adjacency computation} implies that $\sup_{\sigma \ni (\Delta, v)} \iota(\sigma) = \infty$.
\end{proof}

\setcounter{subsection}{1}
\subsection{Classifying the infinite adjacencies}

In this subsection we prove Theorem \ref{classification result}, our main technical theorem. In Proposition \ref{infinite adjacency phenomenon}, we saw that Delaunay limit triangles intersecting multiple times in a high-modulus horizontal or vertical cylinder give rise to the infinite-adjacency phenomenon. Theorem \ref{classification result} says that this is the only behavior in $\mathcal H_{\mathrm{marked}}(\kappa)$ responsible for the infinite adjacency phenomenon.

Throughout, let $M \in \mathcal H_{\mathrm{marked}}(\kappa)$.

\begin{thm}\label{classification result}
Let $(\Delta, v) \in \mathscr I(\kappa)$. For all but finitely many $\sigma \subset \mathscr I(\kappa)$ with $(\Delta, v) \in \sigma$, every $M \in \mathcal D_\sigma$ has a horizontal or vertical cylinder of modulus greater than $1$.
\end{thm}

Lemma \ref{cylinder without hypothesis} is our main technical lemma for the proof of Theorem \ref{classification result}.

\begin{lem}\label{cylinder without hypothesis}
Suppose that $M$ has two Delaunay limit triangulations $\Delta_1$, $\Delta_2$ such that there are distinct edges $\gamma_1 \in E(\Delta_1)$ and $\gamma_2 \in E(\Delta_2)$ that are not both horizontal or both vertical, and intersect at least $3$ times. Then there is a horizontal or vertical cylinder $C$ in $M$ on which $\gamma_1$ and $\gamma_2$ lie.
\end{lem}

The case where $\gamma_1$ and $\gamma_2$ are either both horizontal or both vertical is treated in Lemma \ref{controlling degenerate edges}. We will prove Lemma \ref{cylinder without hypothesis} by first understanding what happens when $\gamma_1$ and $\gamma_2$ intersect at least 1 time (Lemma \ref{line segment}) and at least 2 times (Lemma \ref{cylinder with hypothesis}).

Note that Lemma \ref{limiting circumsquare} implies that if $\Delta$ is a Delaunay limit triangulation of a translation surface $M$, then every $\gamma \in E(\Delta)$ is inscribed in a maximal $L^\infty$-square.

\begin{defn}
Let $S:(0,h)^2 \to M$ be an $L^\infty$-square. We denote
\[
\dee S \coloneqq \overline{S}(\dee[0,h]^2).
\]
\end{defn}

Following Remark \ref{lifting remark}, we shall make frequent appeal to the universal cover $\Tilde M \twoheadrightarrow M$ in this subsection.

\begin{lem}\label{line segment}
Suppose that a translation surface $M = (X,\omega)$ has Delaunay limit triangulations $\Delta_1$, $\Delta_2$ such that there are distinct edges $\gamma_1 \in E(\Delta_1)$ and $\gamma_2 \in E(\Delta_2)$ that intersect. Let $S_1$ and $S_2$ be maximal $L^\infty$-squares in which $\gamma_1$ and $\gamma_2$ are respectively inscribed. Then the intersection $\dee S_1 \cap \dee S_2$ contains a line segment.
\end{lem}

\begin{proof}
Let us lift the $\gamma_j$ to intersecting saddle connections $\Tilde \gamma_j$ on the universal cover $\Tilde M = (\Tilde X, \Tilde \omega)$ of $M$. Let us also lift each $S_j$ to $\Tilde M$ so that $\gamma_j$ is inscribed in $S_j$. Up to re-indexing, rotation, and reflection, the union of $\Tilde{S_1}$ and $\Tilde{S_2}$ is in one of the four forms shown in Figure \ref{four square configurations} (cf. \cite{gueritaud}, Figure 5).

\begin{figure}
\begin{center}
\includegraphics[scale=0.365]{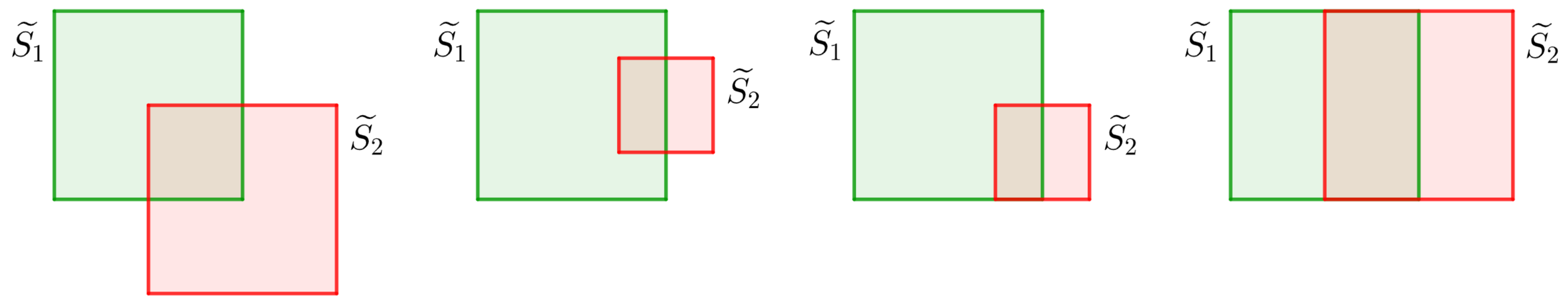}
\end{center}
\caption{The four possible configurations of $\Tilde S_1$ and $\Tilde S_2$}
\label{four square configurations}
\end{figure}

Note that in the latter two forms, the intersection $\dee\Tilde{S_1} \cap \dee\Tilde{S_2}$ contains a line segment.  Since $\Tilde S_j^*\Tilde \omega = dz$, singularities of $\Tilde S_j$ may only appear on $\dee \Tilde S_j \setminus \Tilde{S_{j\pm1}}$. In the former two forms, the convex hulls of $\dee \Tilde S_1 \setminus \Tilde{S_2}$ and of $\dee \Tilde{S_2} \setminus \Tilde S_1$ intersect in a single line segment, contradicting the fact that $\Tilde \gamma_1$ and $\Tilde \gamma_2$ are distinct. We conclude that we must be in one of the latter two situations, and hence the intersection $\dee \Tilde S_1 \cap \dee \Tilde{S_2}$ contains a line segment. Returning to $M$ via the covering map, we conclude that the intersection $\dee S_1 \cap \dee S_2$ contains a line segment.
\end{proof}

\begin{lem}\label{cylinder with hypothesis}
Suppose that $M$ has Delaunay limit triangulations $\Delta_1$, $\Delta_2$ such that there are distinct edges $\gamma_1 \in E(\Delta_1)$ and $\gamma_2 \in E(\Delta_2)$ that are not both horizontal or both vertical, and intersect at least $2$ times.

Again for $j=1,2$, let $S_j$ be a maximal $L^\infty$-square in which $\gamma_j$ is inscribed. Suppose without loss of generality that $\mathrm{height}(S_1) \geq \mathrm{height}(S_2)$. Let $\Tilde{\gamma_1}$ be a lift of $\gamma_1$ to the universal cover $\Tilde M \twoheadrightarrow M$, and let $\Tilde{S_1}$ be a lift of $S_1$ in which $\Tilde{\gamma_1}$ is inscribed. Now let $\Tilde{S_2^1}$ and $\Tilde{S_2^2}$ be  two distinct lifts of $S_2$, so that there are two distinct lifts $\Tilde{\gamma_2^1}$ and $\Tilde{\gamma_2^2}$ of $\gamma_2$, with each $\Tilde{\gamma_2^i}$ inscribed in $\Tilde{S_2^i}$, and so that each $\Tilde{\gamma_2^i}$ intersects $\Tilde \gamma_1$. By Lemma \ref{line segment}, each $\dee\Tilde{S_2^i} \cap \dee\Tilde S_1$ contains a line segment $\alpha^i$.

If $\alpha^1$ and $\alpha^2$ lie on the same side of $\Tilde S_1$, then there is a horizontal or vertical cylinder $C$ in $M$ on which $\gamma_1$ and $\gamma_2$ lie.

\end{lem}

\begin{proof}
Suppose without loss of generality that $\alpha^1$ and $\alpha^2$ lie on the bottom side of $\Tilde S_1$. Since each $\Tilde{S_2^i} \not \subset \Tilde{S_1}$, each $\alpha^i$ meets the bottom left or right vertex of $\Tilde{S_1}$. We have two cases.
\smallskip

\noindent\textbf{Case 1:} Without loss of generality, $\alpha^1$ and $\alpha^2$ both meet the bottom left vertex of $\Tilde{S_1}$. See Figure \ref{2 intersections case 1}.

\begin{figure}
\begin{center}
\includegraphics[scale=1.05]{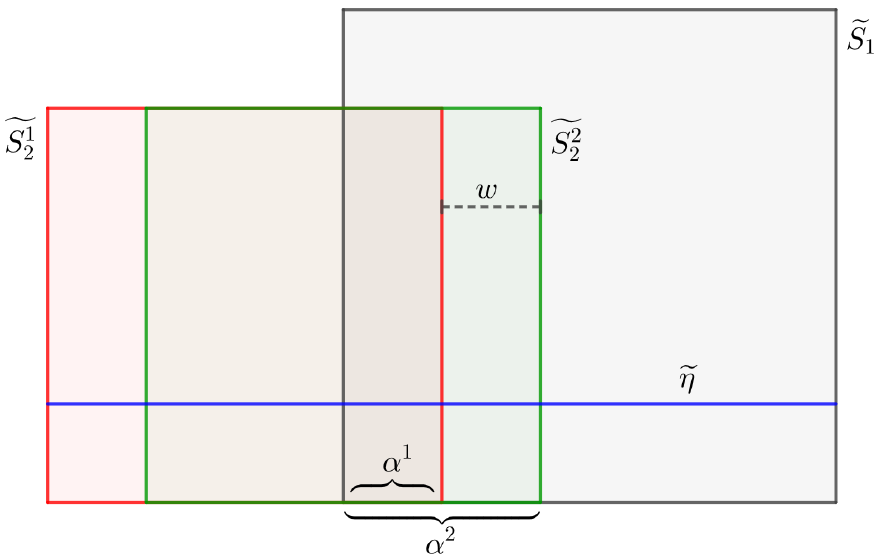}
\end{center}
\caption{Case 1 of Lemma \ref{cylinder with hypothesis}}
\label{2 intersections case 1}
\end{figure}

Let us consider a horizontal line segment $\Tilde \eta$ in $\Tilde{S_2^1} \cup \Tilde{S_1} \cup \Tilde{S_2^2}$. The segment $\Tilde \eta$ must project to a closed loop $\eta$ in $M$; in particular, the intersection points of $\Tilde \eta$ with the right-hand sides of $\Tilde S_2^1$ and $\Tilde S_2^2$ must project to the same point of $M$. Therefore $\eta$ is a core curve of some cylinder $C$ in $M$. Let $w$ be the distance between the right-hand sides of the $\Tilde{S_2^i}$'s. Observe that $\mathrm{width}(C) \leq w < \mathrm{width}(S_1)$. Therefore $\mathrm{height}(C) = \mathrm{height}(S_1) = \mathrm{height}(S_2)$, for otherwise each singularity on the top side of $C$ would lift to a point on the interior of $\Tilde{S_1}$. Therefore the images of $S_1$ and $S_2$ are contained in $C$, and so $\gamma_1$ and $\gamma_2$ lie on $C$, as desired.
\smallskip

\noindent\textbf{Case 2:} Without loss of generality, $\alpha^1$ meets the bottom left vertex of $\Tilde{S_1}$, and $\alpha^2$ meets the bottom right vertex. See Figure \ref{2 intersections case 2}.

We will show that $\mathrm{height}(S_1) = \mathrm{height}(S_2)$. Then we will be done, because forming a horizontal line segment $\Tilde \eta$ in $\Tilde{S_2^1} \cup \Tilde{S_1} \cup \Tilde{S_2^2}$ again produces a cylinder $C$ in $M$, and this cylinder must contain the images of $S_1$ and $S_2$ because these squares have the same height and contain no singularities on their interiors. Suppose that $\mathrm{height}(S_1) > \mathrm{height}(S_2)$.

\begin{figure}
\begin{center}
\includegraphics[scale=1.05]{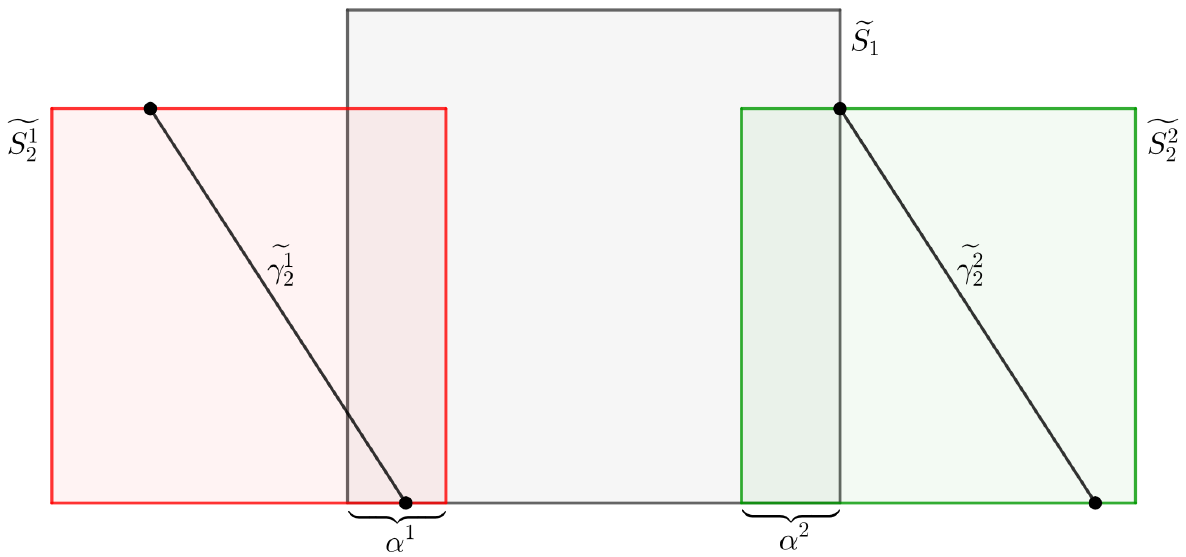}
\end{center}
\caption{Case 2 of Lemma \ref{cylinder with hypothesis}}
\label{2 intersections case 2}
\end{figure}

Observe that $\gamma_2$ can be neither horizontal nor vertical, because then it would not meet the interior of $S_1$, and hence not be able to intersect $\gamma_1$. Without loss of generality, assume $\gamma_2$ has negative slope. But then $\Tilde \gamma_2^2$ cannot meet the interior of $\Tilde{S_1}$, because $\Tilde{S_1}$ can have no singularities on its interior. This contradicts the fact that each lift of $\gamma_2$ intersects $\Tilde{\gamma_1} \subset \Tilde{S_1}$, and so we conclude that $\mathrm{height}(S_1) = \mathrm{height}(S_2)$.
\end{proof}

\begin{proof}[Proof of Lemma \ref{cylinder without hypothesis}]
We begin as in the hypothesis of Lemma \ref{cylinder with hypothesis}. For $j=1,2$, let $S_j$ be a maximal $L^\infty$-square in which $\gamma_j$ is inscribed. Suppose without loss of generality that $\mathrm{height}(S_1) \geq \mathrm{height}(S_2)$.

Let $\Tilde{\gamma_1}$ be a lift of $\gamma_1$ to the universal cover $\Tilde M \twoheadrightarrow M$, and let $\Tilde{S_1}$ be a lift of $S_1$ in which $\Tilde{\gamma_1}$ is inscribed. Now let $\Tilde {S_2^1}$, $\Tilde{S_2^2}$, and $\Tilde{S_2^3}$ be three distinct lifts of $S_2$, so that there are three distinct lifts $\Tilde{\gamma_2^1}$, $\Tilde{\gamma_2^2}$ and $\Tilde{\gamma_2^3}$ of $\gamma_2$, with each $\Tilde{\gamma_2^i}$ inscribed in $\Tilde{S_2^i}$, and so that each $\Tilde{\gamma_2^i}$ intersects $\Tilde \gamma_1$. By Lemma \ref{line segment}, each $\dee\Tilde{S_2^i} \cap \dee\Tilde S_1$ contains a line segment $\alpha^i$. By Lemma \ref{cylinder with hypothesis}, we are done if we can show that two of the $\alpha^i$ lie on the same side of $\Tilde{S_1}$. If $\mathrm{height}(S_1) = \mathrm{height}(S_2)$, then this is immediate.

Suppose $\mathrm{height}(S_1) > \mathrm{height}(S_2)$, and suppose for the sake of contradiction that each $\alpha^i$ lies on a different side of $\Tilde{S_1}$. If three distinct sides of $S_2$  lie in the interior of $S_1$, then $S_2$ can only have singularities on one of its sides, and hence $\gamma_2$ cannot meet the interior of $S_1$, which is necessary for $\gamma_2$ to intersect $\gamma_1$. Up to rotation and reflection, there is only one configuration of the four squares $\Tilde{S_1}$, $\Tilde{S_2^i}$ so that only two sides of $S_2$ lie on the interior of $S_1$, as shown in Figure \ref{3 intersections}.

\begin{figure}
\begin{center}
\includegraphics[scale=1.2]{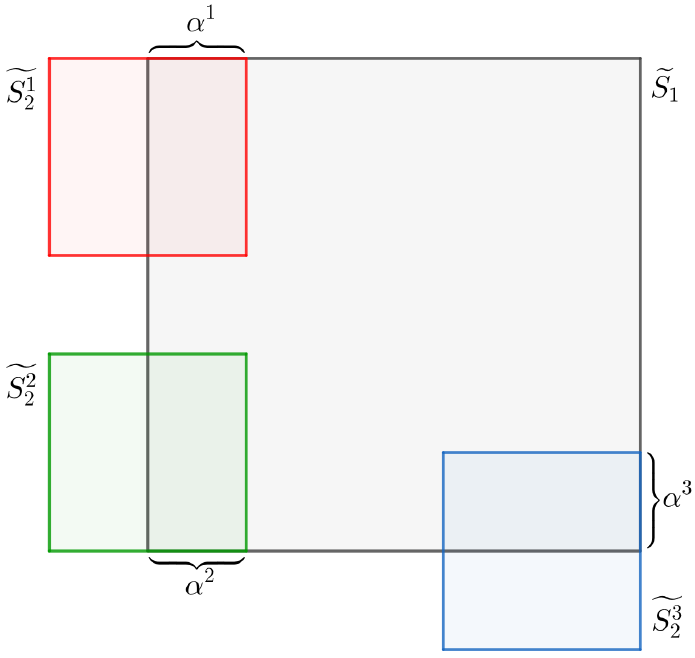}
\end{center}
\caption{Unique configuration with $\alpha^i$ on different sides and with two sides of $S_2$ in $S_1$}
\label{3 intersections}
\end{figure}

In this configuration, the fact that the interior of an $L^\infty$-square contains no singularities constrains the endpoints of each $\Tilde{\gamma_2^i}$ to lie in the bottom left quarter of $\dee\Tilde{S_2^i}$, and hence they cannot meet the interior of $\Tilde{S_1}$, a contradiction. We conclude that two of the $\alpha^i$ must lie on the same side of $\Tilde{S_1}$, and so we are done.
\end{proof}

\begin{lem}\label{controlling degenerate edges}
Suppose that $M$ has Delaunay limit triangulations $\Delta_1$, $\Delta_2$ such that there are distinct edges $\gamma_1 \in E(\Delta_1)$ and $\gamma_2 \in E(\Delta_2)$ that intersect, and are either both horizontal or both vertical. Then $\iota(\gamma_1, \gamma_2) \leq 2$.
\end{lem}

\begin{proof}
Suppose without loss of generality that $\gamma_1$ and $\gamma_2$ are horizontal, and let $S_1$ and $S_2$ be $L^\infty$-squares in which $\gamma_1$ and $\gamma_2$ are respectively inscribed. For the sake of clarity, let $g_1$ and $g_2$ denote these horizontal representatives of the topological arcs $\gamma_j \in E(\Delta_j)$. Recall from Remark \ref{isotopy class remark} that since the $g_j$ may contain points of $\Sigma$, they may not be isotopic rel $\Sigma$ to the arcs $\gamma_j$. Since the $g_j$ arise via limits of $L^\infty$-triangle edges, there is a direction (up or down) in which $g_j$ may be homotoped, rel endpoints, to an arc $c_j$ in the interior of $S_j$, so that $c_j$ is isotopic rel $\Sigma$ to $\gamma_j$. For each $j$, call this direction the \emph{flexible direction} for $g_j$. We have two cases. Either the flexible directions for $g_1$ and $g_2$ are the same or different. If they are different, then homotoping them into the interior of their respective $S_j$ demonstrates that $\iota(\gamma_1, \gamma_2) = 0$.

Now suppose that the $g_j$ have the same flexible direction. Suppose without loss of generality that $\mathrm{length}(g_2) \geq \mathrm{length}(g_1)$. Then there is some $0 < \e \ll 1$ such that each $g_j$ can be homotoped to a circular arc $c_j$ of height less than $\e$ in the interior of $S_j$. Let us say that the \emph{curvature} of $c_j$ the inverse of the radius of the circle of which it is an arc.

We may choose the arcs $c_j$ such that the length of $g_j$ is smaller than the diameter of the circle of which $c_j$ is an arc. We may further suppose that $c_2$ is taller than $c_1$, and also has greater curvature. See Figure \ref{circular arcs}.
\begin{figure}
\begin{center}
\includegraphics[scale=3]{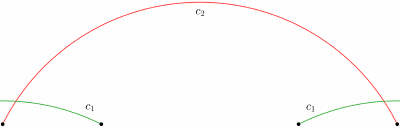}
\end{center}
\caption{The greatest possible number of intersections between $c_1$ and $c_2$}
\label{circular arcs}
\end{figure}
Since $c_1$ and $c_2$ have different heights, if these arcs meet, then they meet transversely. If these arcs meet more than 2 times, then they must form a bigon. However, a wide, tall circular arc cannot form a bigon with a short, squat circular arc of lesser curvature. Thus $c_1$ and $c_2$ cannot form a bigon. Therefore $\iota(\gamma_1,\gamma_2)$ is equal to the number of intersection points of these circular arcs, and we conclude that $\iota(\gamma_1, \gamma_2) \leq 2$.
\end{proof}

\begin{lem}\label{modulus bound}
Let $M$ be a translation surface, and let $\Delta_1$ and $\Delta_2$ be two Delaunay limit triangulations of $M$. Suppose there are edges $\gamma_1 \in E(\Delta_1)$ and $\gamma_2 \in E(\Delta_2)$ such that $\iota(\gamma_1, \gamma_2) \geq 3$, so that they lie in a horizontal or vertical cylinder $C$. Then
\[
\mathrm{mod}(C) > \frac{\iota(\gamma_1,\gamma_2)-1}{2}
\]
\end{lem}
\begin{proof}
Suppose without loss of generality that $C$ is horizontal, and let $w$ and $h$ be its width and height, respectively. Since $\gamma_1$ and $\gamma_2$ are straight line segments inscribed in squares with their endpoints lying on the top and bottom sides of the squares, they must have slope of magnitude no less than 1. Letting $\Tilde{\gamma_1}$ and $z_i$ be as in Lemma \ref{displacement formula}, it follows from that lemma that the horizontal distance from each $z_i$ to $z_{i+1}$ is at least $\tfrac w2$. There are $k-1$ consecutive pairs $(z_i,z_{i+1})$ and so $(k-1)\tfrac w2 < \mathrm{width}(\Tilde{\gamma_1}) \leq h$. The first inequality is strict because no $z_i$ is an endpoint of $\Tilde{\gamma_1}$. Therefore $\mathrm{mod}(C) = \tfrac hw > \tfrac{k-1}{2}$.
\end{proof}

\begin{lem}\label{intersection bound}
Let $K \geq 2$ be an integer, and let $(\Delta, v) \in \mathscr I(\kappa)$. For all but finitely many $\sigma \subset \mathscr I(\kappa)$ with $(\Delta, v) \in \sigma$, we have $\iota(\sigma) \geq K$.
\end{lem}

\begin{proof}
This is a consequence of the fact that for any triangulation $\Delta$ of $(S_g, \Sigma)$, there are only finitely many $v \in \{-1,1\}^{2|E(\Delta)|}$, as well as only finitely many other triangulations $\Delta'$ such that $\max \{\iota(\gamma, \gamma') \st  \gamma \in E(\Delta), \gamma' \in E(\Delta')\} < K$.
\end{proof}

\begin{proof}[Proof of Theorem \ref{classification result}]
Lemma \ref{intersection bound} gives us that all but finitely many $\sigma \subset \mathscr I(\kappa)$ with $(\Delta, v) \in \sigma$ have $\iota(\sigma) \geq 3$. Lemmas \ref{cylinder without hypothesis} and \ref{controlling degenerate edges} give us that for every $\sigma$ with $\iota(\sigma) \geq 3$, there is a horizontal or vertical cylinder in every $M \in \mathcal D_\sigma$, and Lemma \ref{modulus bound} gives us that such a cylinder must have modulus greater than $1$. Therefore we are done.
\end{proof}

\section{The isodelaunay complex}\label{complex section}

In this section, we introduce the isodelaunay complex $\mathcal I(\kappa)$ and prove Theorem \ref{intro main}.

\begin{defn}
We denote by $\mathcal T(\kappa) \subset \mathcal H_{\mathrm{marked}}(\kappa)$ the \emph{triple intersection locus}
\[
\mathcal T(\kappa) \coloneqq \bigcup_{\substack{\sigma \subset \mathscr I(\kappa) \\ \iota(\sigma) \geq 3}} \mathcal D_\sigma.
\]
\end{defn}

Theorem \ref{homotopy theorem} tells us that $\mathcal T(\kappa)$ may be deleted without changing the homotopy type of $\mathcal H_{\mathrm{marked}}(\kappa)$.

\begin{thm}\label{homotopy theorem}
We have an $\mathrm{MCG}(S_g, \Sigma)$-homotopy equivalence
\[
\mathcal H_{\mathrm{marked}}(\kappa) \simeq_{\mathrm{MCG}(S_g, \Sigma)} \mathcal H_{\mathrm{marked}}(\kappa) \setminus \mathcal T(\kappa).
\]
\end{thm}

\begin{defn}
Let $\mathcal H_{\mathrm{marked}}^{\mathrm{mod}\leq1}(\kappa)$ denote the subset of $\mathcal H_{\mathrm{marked}}(\kappa)$ such that for every $M \in \mathcal H_{\mathrm{marked}}^{\mathrm{mod}\leq1}(\kappa)$, every cylinder in $M$ has modulus at most $1$.
\end{defn}

\begin{rem}
Recall from the proof of Theorem \ref{classification result} that Lemmas \ref{cylinder without hypothesis}, \ref{controlling degenerate edges}, and \ref{modulus bound} together show that $\mathcal H_{\mathrm{marked}}^{\mathrm{mod}\leq1}(\kappa) \subset \mathcal H_{\mathrm{marked}}(\kappa) \setminus \mathcal T(\kappa)$.
\end{rem}

We will prove Theorem \ref{homotopy theorem} by showing that $\mathcal H_{\mathrm{marked}}(\kappa)$ and $\mathcal H_{\mathrm{marked}}(\kappa) \setminus \mathcal T(\kappa)$ both have $\mathcal H_{\mathrm{marked}}^{\mathrm{mod}\leq1}(\kappa)$ as an $\mathrm{MCG}(S_g, \Sigma)$-deformation retract.

\begin{defn}[Modulus-shrinking homotopy]
For a translation surface $M \in \mathcal H_{\mathrm{marked}}(\kappa)$, let $C_1,\dots,C_k$ be all the cylinders in $M$ whose moduli $m_i$ are greater than $1$. For $0 \leq t \leq 1$, we define
\[
H(M,t) \coloneqq a^{C_1}_{(1-t) + t/m_1} \circ \cdots \circ a^{C_k}_{(1-t) + t/m_k}(M).
\]
Note that this is well-defined and independent of the indexing of the $C_i$, because in any translation surface, cylinders of modulus greater than $1$ are disjoint from each other. Also note that $H$ is an $\mathrm{MCG}(S_g, \Sigma)$-homotopy.
\end{defn}

\begin{lem}\label{non-decreasing intersection numbers}
Let $C$ be a horizontal or vertical cylinder on a translation surface $M$. Then $t \mapsto \iota(a^C_t(M))$ is non-decreasing for $t > 1/\mathrm{mod}(C)$.
\end{lem}

\begin{proof}
Assume without loss of generality that $C$ is horizontal. We will show that $t \mapsto \iota(a_t^C(M); C)$ is non-decreasing for $t > 1/\mathrm{mod}(C)$. Since the part of any triangulation lying in the complement of $C$ in $M$ never changes as $t$ increases, the desired result will then follow.

Let us write $h = \hei(C)$ and $a^C_t(M) = M_t = (X_t,\omega_t)$, and let us consider the universal cover $\Tilde M_t = (\Tilde X_t, \Tilde \omega_t) \twoheadrightarrow M_t$. For all $t$, let us parametrize the preimage $\Tilde C$ of $C$ in $\Tilde M_t$ with the strip $\{0 < \Im(z) < h\}$, such that $\Tilde \omega_t|_{\Tilde C}$ is expressed as $dx + itdy$ on this strip. In these coordinates, maximal $L^\infty$-squares are given by rectangles of width $th$ and height $1$.

Let $t > 1/\mathrm{mod}(C)$, so that $C$ has modulus greater than $1$ on $M_t$. By Lemma \ref{tilted triangulations are optimal}, we have
\[
\iota(M_t; C) = \max \{\iota(\gamma, \gamma') \st \gamma \in E(\Delta_M^+), \gamma' \in E(\Delta_M^-) \text{ lying in }C\}.
\]
Let us consider a maximal $L^\infty$-square $S_t$ in $\Tilde C$, one of whose vertices is a singularity $x$ of $\Tilde M_t$, and consider the saddle connection $\gamma_{x,S_t}$ that minimizes $|\mathrm{slope}(\gamma_{x,S_t})|$ among saddle connections inscribed in $S_t$ and starting at $x$. We are considering here the slope in our non-conformal, $t$-dependent parametrization, but we will make a conclusion about the amount of times that saddle connections intersect, which is independent of parametrization.

We claim that as $t$ increases, $|\mathrm{slope}(\gamma_{x,S_t})|$ cannot increase. Indeed, one endpoint of the saddle connection $\gamma_{x,S_t}$ is always $x$, and the other endpoint is the singularity furthest from $x$ that meets $S_t$ on the opposite side. In our parametrization, all the singularities have coordinates independent of $t$, and the rectangle $S_t$ simply becomes wider. Hence we see that the only time $|\mathrm{slope}(\gamma_{x,S_t})|$ will change is when $S_t$ meets a new singularity, in which case $|\mathrm{slope}(\gamma_{x,S_t})|$ can only decrease.

By Lemma \ref{tilted triangulations are optimal}, every pair of edges realizing the maximum $\iota(M; C)$ are of the form $\gamma_{x,S_t}$, $\gamma_{x',S'_t}$ with slopes of opposite sign. Thus as $t$ increases, the minimal $y$-coordinate among points of $\gamma_{x,S_t} \cap \gamma_{x',S'_t}$ is non-increasing. By Lemma \ref{displacement formula}, the vertical distance between a consecutive pair of points in $\gamma_{x,S_t}$ and $\gamma_{x',S'_t}$ is also non-increasing as $t$ increases. Therefore $t \mapsto \iota(\gamma_{x,S_t}, \gamma_{x',S'_t})$ is non-decreasing. Even though our parametrization is non-conformal and $t$-dependent, this last fact is independent of parametrization, and so we conclude that $\iota(M_t; C)$ is non-decreasing, as desired.
\end{proof}

\begin{lem}\label{homotopy restriction}
For every $0 \leq t \leq 1$ and $M \in \mathcal H_{\mathrm{marked}}(\kappa)$, we have $H(M, t) \in \mathcal T(\kappa)$ only if $M \in \mathcal T(\kappa)$.
\end{lem}

\begin{proof}
Suppose $M \notin \mathcal T(\kappa)$, so that $\iota(M) < 3$. If $M$ has no horizontal or vertical cylinders of modulus greater than $1$, then $H(M,t)$ has no such cylinders either, and hence $\iota(H(M,t)) < 3$ for every $0 \leq t \leq 1$ by Lemma \ref{cylinder without hypothesis}. If $M$ does have such cylinders, then Lemma \ref{non-decreasing intersection numbers} again gives $\iota(H(M,t)) < 3$ for every $0 \leq t \leq 1$. In each case, we conclude $H(M,t) \notin \mathcal T(\kappa)$.
\end{proof}

\begin{proof}[Proof of Theorem \ref{homotopy theorem}]
Since $H$ is a deformation retraction of $\mathcal H_{\mathrm{marked}}(\kappa)$ onto $\mathcal H_{\mathrm{marked}}^{\mathrm{mod}\leq1}(\kappa)$, we have $\mathcal H_{\mathrm{marked}}(\kappa) \simeq \mathcal H_{\mathrm{marked}}^{\mathrm{mod}\leq1}(\kappa)$. By Lemma \ref{homotopy restriction}, we also see that $H$ is a deformation retraction of $\mathcal H_{\mathrm{marked}}(\kappa) \setminus \mathcal T(\kappa)$ onto $\mathcal H_{\mathrm{marked}}^{\mathrm{mod}\leq1}(\kappa)$. Therefore we also have $\mathcal H_{\mathrm{marked}}(\kappa) \setminus \mathcal T(\kappa) \simeq \mathcal H_{\mathrm{marked}}^{\mathrm{mod}\leq1}(\kappa)$. Since $H$ is an $\mathrm{MCG}(S_g, \Sigma)$-homotopy, we conclude that
\[
\mathcal H_{\mathrm{marked}}(\kappa) \simeq_{\mathrm{MCG}(S_g,\Sigma)} \mathcal H_{\mathrm{marked}}^{\mathrm{mod}\leq1}(\kappa) \simeq_{\mathrm{MCG}(S_g,\Sigma)} \mathcal H_{\mathrm{marked}}(\kappa) \setminus \mathcal T(\kappa).
\]
\end{proof}

\begin{defn}
Let $\mathscr D(\kappa)_{\mathrm{finite}}$ be the covering of $\mathcal H_{\mathrm{marked}}(\kappa) \setminus \mathcal T(\kappa)$ by the closed subsets $\mathcal F(\Delta, v) \coloneqq \mathcal D({\Delta, v}) \setminus \mathcal T(\kappa)$, where $(\Delta,v) \in \mathscr I(\kappa)$. Let $\mathcal F_\sigma \coloneqq \bigcap_{(\Delta,v)} \mathcal F(\Delta, v)$ for $\sigma \subset \mathscr I(\kappa)$.
\end{defn}

\begin{thm}\label{main theorem}
The complex $\mathcal N(\mathscr D(\kappa)_{\mathrm{finite}})$ is locally finite and $\mathrm{MCG}(S_g,\Sigma)$-invariant, and we have a homotopy equivalence
\[
\mathcal H(\kappa) \simeq \mathcal N(\mathscr D(\kappa)_{\mathrm{finite}})/\mathrm{MCG}(S_g,\Sigma).
\]
\end{thm}

\begin{proof}
We first verify Lemma \ref{nerve hypotheses} with $\mathscr D_{\mathrm{finite}}(\kappa)$ in place of $\mathscr D(\kappa)$. Hypotheses (\ref{good cover}) and (\ref{finite isotropy}) are immediate from that lemma. To see (\ref{equivariant contractibility}) observe that $\mathrm{MCG}(S_g, \Sigma)$ must have a fixed point $M_{\mathrm{avg}}$ in the relative interior\footnote{By \ref{convex intersections}, $\mathcal D_\sigma$ is homeomorphic to a convex subset of $\R^{4g+2n-2}$, whose affine span is some affine $k$-plane $A$. The \emph{relative interior} of $\mathcal D_\sigma$ is its interior as a subset $A$.} of $\mathcal D_\sigma$, since this relative interior is convex. Since $\mathcal F_\sigma = \mathcal D_\sigma \setminus \mathcal T(\kappa)$ is star-shaped with respect to $M_{\mathrm{avg}}$, we again have that the straight-line homotopy $(M,t) \mapsto (1-t)M + tM_{\mathrm{avg}}$ is an $\mathrm{MCG}(S_g, \Sigma)$-deformation retraction of $\mathcal F_\sigma$ onto $M_\mathrm{avg}$. We once again defer the proof of (\ref{equivalent refinement check}) to Lemma \ref{D(k) is a refinement}.

We may now apply Theorem \ref{equivariant nerve lemma} to obtain an $\mathrm{MCG}(S_g,\Sigma)$-homotopy equivalence
\[
\mathcal H_{\mathrm{marked}}(\kappa) \setminus \mathcal T(\kappa) \simeq_{\mathrm{MCG}(S_g,\Sigma)} \mathcal N(\mathscr D(\kappa)_{\mathrm{finite}}).
\]
Together with Theorem \ref{homotopy theorem}, this gives
\[
\mathcal H_{\mathrm{marked}}(\kappa) \simeq_{\mathrm{MCG}(S_g,\Sigma)} \mathcal N(\mathscr D(\kappa)_{\mathrm{finite}}).
\]
By $\mathrm{MCG}(S_g,\Sigma)$-equivariance, this homotopy equivalence descends to the quotient by $\mathrm{MCG}(S_g,\Sigma)$, and so we have
\[
\mathcal H_{\mathrm{marked}}(\kappa)/\mathrm{MCG}(S_g,\Sigma) = \mathcal H(\kappa) \simeq \mathcal N(\mathscr D(\kappa)_{\mathrm{finite}})/\mathrm{MCG}(S_g,\Sigma).
\]
Local finiteness of $\mathcal N(\mathscr D(\kappa)_{\mathrm{finite}})$ follows from Lemma \ref{intersection bound}.
\end{proof}

\begin{rem}\label{orbifold barycentric}
Recall that if a group $G$ acts on a simplicial complex $\mathcal X$ by simplicial automorphisms, the quotient space $\mathcal X/G$ may not inherit a CW structure from $\mathcal X$. Nonetheless, letting $\mathcal X'$ denote the barycentric subdivision of $\mathcal X$, the quotient $\mathcal X'/G$ is a CW complex. Furthermore, by taking another barycentric subdivision, we are guaranteed that $\mathcal X''/G$ is a simplicial complex. See e.g. Proposition III.1.1 of \cite{bredon}.
\end{rem}

\begin{defn}
Let us write $\mathcal I(\kappa) \coloneqq \mathcal N(\mathscr D(\kappa)_{\mathrm{finite}})''/\mathrm{MCG}(S_g,\Sigma)$, and let us call this simplicial complex the \emph{isodelaunay complex}.
\end{defn}

\begin{proof}[Proof of Theorem \ref{intro main}]
The first claim is the content of Theorem \ref{main theorem}. We now show that $\mathcal I(\kappa)$ is finite. Since $\mathcal N(\mathscr D_{\mathrm{finite}}(\kappa))$ is locally finite, it suffices to show that $\mathcal N(\mathscr D_{\mathrm{finite}}(\kappa))$ has finitely many $\mathrm{MCG}(S_g, \Sigma)$-orbits of vertices, i.e. there are finitely many isodelaunay regions up to the action of the mapping class group. Observe that these orbits $\mathrm{MCG}(S_g, \Sigma)\mathcal D(\Delta, v)$ are in one-to-one correspondence with equivalence classes of Delaunay data $(\Delta, v)$ up to homeomorphism. Since there are, up to homeomorphism, finitely many triangulations $\Delta$ of $S_g$ with vertices in $\Sigma$, and finitely many $v \in \{-1, 1\}^{2|E(\Delta)|}$, we are done.
\end{proof}

\setcounter{subsection}{1}
\subsection{Computability of the isodelaunay complex}

In this subsection we briefly outline a method for constructing $\mathcal I(\kappa)$ explicitly, and the computational problems involved in doing so.

\subsubsection*{Enumerating Delaunay data up to homeomorphism}
Recall that vertices of the quotient $\mathcal N(\mathscr D_{\mathrm{finite}}(\kappa))/\mathrm{MCG}(S_g,\Sigma)$ can be identified with equivalence classes of Delaunay data $(\Delta, v)$ up to homeomorphism, and that there are finitely many of these equivalence classes. They can be picked out from the larger collection of homeomorphism types of veering triangulations using the software package \texttt{veerer} \cite{veerer}. This software package is designed to handle veering triangulations of translation surfaces, and implements an algorithm that detects whether a given homeomorphism type of veering triangulation arises as an $L^\infty$-Delaunay triangulation. Hence we may identify the vertices of $\mathcal N(\mathscr D_{\mathrm{finite}}(\kappa))/\mathrm{MCG}(S_g,\Sigma)$ from among the set of all homeomorphism types of veering triangulations.

\subsubsection*{Finding the simplices incident to a vertex of $\mathcal N(\mathscr D_{\mathrm{finite}}(\kappa))$}
The edges of $\mathcal N(\mathscr D_{\mathrm{finite}}(\kappa))$ incident to a given vertex $w = \mathcal D(\Delta, v)$ are given by pairs $\{(\Delta, v), (\Delta', v')\}$ of Delaunay data such that $\iota(\Delta, \Delta') \coloneqq \max \{\iota(\gamma, \gamma') \st  \gamma \in E(\Delta), \gamma' \in E(\Delta')\} < 3$ and $\mathcal D(\Delta, v) \cap \mathcal D(\Delta', v') \ne \emptyset$. We describe how to find all edges incident to $w$.

The first problem is to find all triangulations $\Delta'$ such that $\iota(\Delta, \Delta') < 3$, and to find all coefficient vectors $v'$ compatible with these triangulations. Let $\sigma(\Delta, v) \subset \mathscr I(\kappa)$ be the set of all such $(\Delta', v')$. Finding $\sigma(\Delta, v)$ is a computational problem in surface topology. The second problem is, for each $(\Delta', v') \in \sigma(\Delta, v)$, to compute $\mathcal D(\Delta, v) \cap \mathcal D(\Delta', v')$. This is a problem in linear programming: once a system of coordinates is fixed, this is the problem of deciding whether the nonstrict veering and quadrilateral inequalities for $(\Delta, v)$ and $(\Delta', v')$ have a simultaneous solution.

The $n$-simplices of $\mathcal N(\mathscr D_{\mathrm{finite}}(\kappa))$ incident to $w$ are given by mapping class group orbits of sets $\{(\Delta, v) = (\Delta_1,v_1), \dots, (\Delta_n, v_n)\}$ of Delaunay data that satisfy $\iota(\Delta_i, \Delta_j) < 3$ pairwise, such that $\bigcap_{i=1}^n\mathcal D(\Delta_i, v_i) \ne \emptyset$. To find these $n$-simplices, the only additional problem is to decide whether the nonstrict veering and quadrilateral inequalities for each size $n$ subset of $\sigma(\Delta, v)$ have a simultaneous solution.

\subsubsection*{Passing to the quotient}
The previous two steps find every vertex $\mathrm{MCG}(S_g, \Sigma)\mathcal D(\Delta, v)$ of $\mathcal N(\mathscr D_{\mathrm{finite}}(\kappa))/\mathrm{MCG}(S_g, \Sigma)$ and all the simplices incident to some preimage $w = \mathcal D(\Delta, v)$ of this vertex in $\mathcal N(\mathscr D_{\mathrm{finite}}(\kappa))$. It remains to understand how these simplices are identified with each other under the quotient mapping. That is to say, given $\sigma^j = \{(\Delta_1^j,v_1^j), \dots, (\Delta_n^j, v_n^j)\}$ for $j = 1,2$, we want to know whether there exists some $g \in \mathrm{MCG}(S_g, \Sigma)$ such that $g\mathcal D_{\sigma^1} = \mathcal D_{\sigma^2}$. This is again a computational problem in surface topology: we must determine whether there is a self-homeomorphism (i.e. re-marking) of the surface $(S_g, \Sigma)$ that takes the configuration $\sigma^1$ of $n$ simultaneous veering triangulations to the configuration $\sigma^2$.

Among the possibilities to be aware of here is that of self-adjacency. That is to say, we may have $g(\mathcal D(\Delta_1, v_1) \cap \mathcal D(\Delta_2, v_2)) = \mathcal D(\Delta_1, v_1) \cap \mathcal D(\Delta_3, v_3)$, such that all three $\mathcal D(\Delta_j, v_j)$ lie in the same $\mathrm{MCG}(S_g,\Sigma)$-orbit. It turns out that this occurs even in the case of $\mathcal H_{\mathrm{marked}}(0)$. This gives rise to a self-loop at the vertex $\mathrm{MCG}(S_g, \Sigma)\mathcal D(\Delta_1, v_1)$ of $\mathcal N(\mathscr D_{\mathrm{finite}}(\kappa))/\mathrm{MCG}(S_g,\Sigma)$. A further possibility is that of automorphisms of $n$-cells, i.e. the case where $\sigma^1 = \sigma^2$. It is for reasons such as these that we must take the second barycentric subdivision $\mathcal N(\mathscr D_{\mathrm{finite}}(\kappa))''/\mathrm{MCG}(S_g,\Sigma) = \mathcal I(\kappa)$ in order to ensure that this quotient has the structure of a simplicial complex.

\appendix
\section{An equivariant Nerve Lemma}

Except for Lemma \ref{D(k) is a refinement}, this appendix is independent of the rest of the paper. Its purpose is to recall some topological definitions and to prove Theorem \ref{equivariant nerve lemma}. Our proof is a generalization of the methods of Section 5 of \cite{gonzalez-gonzalez}.

In Lemma \ref{D(k) is a refinement}, we produce combinatorial ``regular neighborhoods'' by using the fact that the sets $\overline{\mathcal P(\Delta, v)^\circ}$ are polytopes.

\begin{lem}\label{D(k) is a refinement}
The closed coverings $\mathscr D(\kappa)$ and $\mathscr D(\kappa)_{\mathrm{finite}}$ of $\mathcal H_{\text{marked}}(\kappa)$ and $\mathcal H_{\text{marked}}(\kappa) \setminus \mathcal T(\kappa)$, respectively, are equivalent refinements of $\mathrm{MCG}(S_g, \Sigma)$-invariant open coverings.
\end{lem}

\begin{proof}
We will first endow $\mathcal H_{\text{marked}}(\kappa)$ with the structure of an  infinite $\mathrm{MCG}(S_g, \Sigma)$-invariant simplicial complex such that every $\mathcal D_\sigma$ is a subcomplex. Recall the notation of Lemma \ref{convex intersections}. For each integer $n > 1$, define
\[
\mathcal C_n(\Delta, v) \coloneqq \overline{\mathcal P(\Delta, v)^\circ} \cap \bigcap_{\gamma \in E(\Delta)} \{\mathbf z \in \C^{2g+n-1} \st \tfrac1n \leq |\Re(z_\gamma)| + |\Im(z_\gamma)| \leq n\}.
\]
We claim that $\mathcal C_n(\Delta, v)$ is a compact convex subset of $\mathcal P(\Delta, v)$. It is clear from the defining inequalities that $\mathcal C_n(\Delta, v)$ is compact. Now observe that the veering inequalities imply that $|\Re(z_\gamma)|$ is either always $\Re(z_\gamma)$ or always $-\Re(z_\gamma)$ for every $\mathbf z \in \overline{\mathcal P(\Delta, v)^\circ}$, and similarly for $|\Im(z_\gamma)|$. Therefore for each $\gamma \in E(\Delta)$, the inequalities $\tfrac1n \leq |\Re(z_\gamma)| + |\Im(z_\gamma)| \leq n$ define the convex subset of $\overline{\mathcal P(\Delta, v)^\circ}$ bounded between two real hyperplanes. Finally we show that $\mathcal C_n(\Delta, v) \subset \mathcal P(\Delta, v)$. Since $\mathcal P(\Delta, v) = \overline{\mathcal P(\Delta, v)^\circ} \setminus \mathcal Z_0(\Delta)$, it suffices to observe that $\mathcal C_n(\Delta, v) \cap \mathcal Z_0(\Delta) = \emptyset$, which holds because $\tfrac1n \leq |\Re(z_\gamma)| + |\Im(z_\gamma)|$ implies that $z_\gamma$ is bounded away from $0$ for every $\gamma \in E(\Delta)$.

Since $\mathcal C_n(\Delta, v)$ is a compact set defined by finitely many nonstrict linear inequalities, it is a compact convex polytope, and hence triangulable by finitely many simplices, e.g. via barycentric subdivision. Let us choose these finite triangulations of each $\mathcal C_n(\Delta, v)$ such that if $m < n$, then $\mathcal C_m(\Delta, v) \subset \mathcal C_n(\Delta, v)$ is a subcomplex. Furthermore, let us choose our triangulations such that $\mathcal C_m(\Delta, v) \cap \mathcal C_n(\Delta', v')$ is a subcomplex of both $\mathcal C_m(\Delta, v)$ and $\mathcal C_n(\Delta', v')$ for every $m,n \in \N$. Finally, when $\mathcal D(\Delta', v') = g\mathcal D(\Delta, v)$ for $g \in \mathrm{MCG}(S_g,\Sigma)$, note that the sets $\mathcal C_n(\Delta, v)$ and $\mathcal C_n(\Delta', v')$ are linearly isomorphic. Let us choose our triangulations to be compatible with these linear isomorphisms, i.e. the induced map $g:\mathcal C_n(\Delta, v) \to \mathcal C_n(\Delta', v')$ is also a simplicial isomorphism. 

We now have endowed $\mathcal P(\Delta, v) = \bigcup_{n > 1} \mathcal C_n(\Delta, v)$, and hence also $\mathcal D(\Delta, v) = s_{\Delta, v}(\mathcal P(\Delta, v))$, with the structure of an infinite simplicial complex. By our choice of triangulations, we thus endow $\mathcal H_{\mathrm{marked}}(\kappa) = \bigcup_{(\Delta, v) \in \mathscr I(\kappa)} \mathcal D(\Delta, v)$ with the structure of an infinite $\mathrm{MCG}(S_g, \Sigma)$-invariant simplicial complex such that every $\mathcal D_\sigma$ is a subcomplex.

Now let us take the barycentric subdivision $\mathcal H_{\mathrm{marked}}(\kappa)'$. For each $\sigma \subset \mathscr I(\kappa)$, define $\mathcal A_\sigma$ to be the union of all simplices in $\mathcal H_{\mathrm{marked}}(\kappa)'$ that are disjoint from $\mathcal D_\sigma$, and let $\mathcal U_\sigma = \mathcal H_{\text{marked}}(\kappa)' \setminus \mathcal A_\sigma$. Since our simplicial structure on $\mathcal H_{\mathrm{marked}}(\kappa)'$ is $\mathrm{MCG}(S_g,\Sigma)$-invariant, it follows that $\mathscr U(\kappa) \coloneqq \{\mathcal U(\Delta, v)\}_{(\Delta, v) \in \mathscr I(\kappa)}$ is an $\mathrm{MCG}(S_g, \Sigma)$-invariant open covering of $\mathcal H_{\text{marked}}(\kappa)$. We will show that $\mathscr D(\kappa)$ is an equivalent refinement of $\mathscr U(\kappa)$.

It is straightforward to see that $\mathcal U_\sigma$ is an open neighborhood of $\mathcal D_\sigma$. Since we have taken the barycentric subdivision, it is also straightforward to see that $\mathcal U_\sigma = \bigcap_{(\Delta,v) \in \sigma} \mathcal U(\Delta, v)$. To see that $\mathscr D(\kappa)$ is an equivalent refinement of $\mathscr U(\kappa)$, it remains only to see that each $\mathcal D_\sigma$ is an $\mathrm{MCG}(S_g,\Sigma)_\sigma$-deformation retract of $\mathcal U_\sigma$. This follows from Proposition 2.1 of \cite{gonzalez-gonzalez}.

To realize $\mathscr D_{\mathrm{finite}}(\kappa)$ as an equivalent refinement of some $\mathrm{MCG}(S_g, \Sigma)$-invariant open covering $\mathscr V(\kappa)$ of $\mathcal H_{\mathrm{marked}}(\kappa) \setminus \mathcal T(\kappa)$, we must make the following modifications to the above argument.

For each $\sigma \subset \mathscr I(\kappa)$, define $\mathcal B_\sigma$ to be the union of all simplices in $\mathcal H_{\mathrm{marked}}(\kappa)'$ whose intersection with $\mathcal D_\sigma$ is empty or contained in $\mathcal T(\kappa)$, and let $\mathcal V_\sigma = \mathcal H_{\mathrm{marked}}(\kappa)' \setminus (\mathcal T(\kappa) \cup \mathcal B_\sigma)$. Again, it follows from $\mathrm{MCG}(S_g,\Sigma)$-invariance of our simplicial structure on $\mathcal H_{\mathrm{marked}}(\kappa)'$ that $\mathscr V(\kappa) \coloneqq \{\mathcal V(\Delta, v)\}_{(\Delta, v) \in \mathscr I(\kappa)}$ is an $\mathrm{MCG}(S_g,\Sigma)$-invariant open covering of $\mathcal H_{\mathrm{marked}}(\kappa) \setminus \mathcal T(\kappa)$. It is still straightforward to see that $\mathcal V_\sigma$ is an open neighborhood of $\mathcal F_\sigma = \mathcal D_\sigma \setminus \mathcal T(\kappa)$ in $\mathcal H_{\text{marked}}(\kappa) \setminus \mathcal T(\kappa)$, and that $\mathcal V_\sigma = \bigcap_{(\Delta, v)} \mathcal V(\Delta, v)$ is a consequence of our having taken the barycentric subdivision. Finally, we show that the deformation retraction from $\mathcal V_\sigma$ to $\mathcal F_\sigma$ is given by the same formula as in Section 2 of \cite{gonzalez-gonzalez} as follows.

Every point $x \in \mathcal V_\sigma \subset \mathcal H_{\mathrm{marked}}(\kappa)'$ has a barycentric coordinate representation
\[
x = \sum_{i=1}^r t_id_i + \sum_{j=1}^\rho \tau_jb_j, \quad \quad t_i,\tau_j\geq0 \,\, \forall i,j
\]
where $d_1, \dots, d_r$ are vertices of $\mathcal H_{\mathrm{marked}}(\kappa)$ belonging to $\mathcal D_\sigma$, and $b_1, \dots, b_\rho$ are vertices belonging to $\mathcal B_\sigma$, and $r \geq 1$ and $\rho \geq 0$. We define $\lambda(x) = \sum_{i=1}^r t_i \in [0,1]$ and $f(x) = (\sum_{i=1}^r t_id_i)/\lambda(x)$. Then we have a homotopy
\[
H: \mathcal V_\sigma \times [0,1] \to \mathcal V_\sigma
\]
given by $H(x,s) = (1-s)x + sf(x)$. It is easy to see that $H$ is an $\mathrm{MCG}(S_g,\Sigma)_\sigma$-deformation retraction from $\mathcal V_\sigma$ to $\mathcal F_\sigma$. We conclude that $\mathscr D_{\mathrm{finite}}(\kappa)$ is an equivalent refinement of $\mathscr V(\kappa)$.
\end{proof}

The reader may refer to Sections 2.1 and 4.G of \cite{hatcher} for further discussion of the following definitions in the non-equivariant setting, and to Section 5 of \cite{gonzalez-gonzalez} in the equivariant setting.

\begin{defn}
We say that a simplicial complex $B$ is a $\Delta$\emph{-complex} if every simplex of $B$ is endowed with a total ordering on its vertices. We denote by $[v_0, \dots, v_n]$ the ordered $n$-simplex of $B$ with vertices $v_0 < \cdots < v_n$. In particular, $[v,w]$ denotes the edge $v \to w$. When a group $G$ acts on $B$ in a way compatible with the $\Delta$-complex structure, we denote by $G_v$ the stabilizer of the vertex $v$.
\end{defn}

\begin{defn}
Let $B$ be a $\Delta$-complex. A \emph{complex of spaces} $\mathfrak C$ over $B$ consists of a topological space $\mathfrak C(v)$ for each vertex of $B$ and a map $\mathfrak C[v,w]:\mathfrak C(v) \to \mathfrak C(w)$ for each edge $[v,w]$ of $B$, such that for every $n$-simplex $[v_0,\dots,v_n]$ of $B$, the diagram formed by the maps $\mathfrak C[v_i, v_j]$, where $i<j$, commutes. If a group $G$ acts on $B$, then $\mathfrak C$ is a $G$\emph{-complex of spaces} if there are homeomorphisms $\mathfrak C(v) \to \mathfrak C(gv)$ that are compatible with the group action.

Given complexes of spaces $\mathfrak C$ and $\mathfrak D$ over $B$, a \emph{map} $\mathfrak F: \mathfrak C \to \mathfrak D$ is a collection of maps $\mathfrak F(v): \mathfrak C(v) \to \mathfrak D(v)$ for each vertex $v$ of $B$ such that for each edge $[v,w]$ of $B$, we have $\mathfrak D[v,w] \circ \mathfrak F(v) = \mathfrak F(w) \circ \mathfrak C[v,w]$. If $\mathfrak C$ and $\mathfrak D$ are $G$-complexes, then we require these maps to be equivariant with respect to the action in the obvious way.
\end{defn}

\begin{defn}
We define the \emph{colimit} and \emph{homotopy colimit} of a diagram of spaces $\mathfrak C$ over a $\Delta$-complex $B$. We define $\colim \mathfrak C \coloneqq \left(\coprod_v \mathfrak C(v)\right)/\sim$, where the disjoint union is taken over vertices $v$ of $B$, and $x \sim \mathfrak C[v,w](x)$ for every edge $[v,w]$ of $B$.

We define
\[
\hocolim \mathfrak C \coloneqq \left(\coprod_{[v_0,\dots,v_n]} [v_0,\dots,v_n] \times \mathfrak C(v_0)\right)/\sim,
\]
where the disjoint union is taken over all simplices of $B$.  The equivalence relation is given by identifying along inclusions $[v_0,\dots,\widehat{v_i},\dots,v_n] \times \mathfrak C(v_0) \subset [v_0,\dots,v_n] \times \mathfrak C(v_0)$ for $i > 0$, and identifying $(\alpha, x) \sim (\alpha, \mathfrak C[v_0, v_1](x))$ for $\alpha \in [v_1,\dots,v_n] \subset [v_0,\dots,v_n]$ and $x \in \mathfrak C(v_0)$. We write
\[
[v_0,\dots,v_n] = \left\{\sum_i t_i \cdot u_i \in \R^n \mst t_i \geq 0, \sum_i t_i = 1\right\},
\]
where $\{u_i\}_{i=1}^n$ is a basis for $\R^n$, so that points of $\hocolim \mathfrak C$ are of the form $[(\sum_i t_i \cdot u_i, x)]$. There is a \emph{base projection} map $p_b: \hocolim \mathfrak C \to B$ given by $p_b[(\sum_i t_i \cdot u_i, x)] = \sum_i t_i \cdot u_i$ and a \emph{fiber projection} map $p_f:\hocolim \mathfrak C \to \colim \mathfrak C$ given by $p_f[(\sum_i t_i \cdot u_i, x)] = x$.

Given a map $\mathfrak F: \mathfrak C \to \mathfrak D$ of complexes of spaces over $B$, we clearly have induced maps $\colim \mathfrak F: \colim \mathfrak C \to \colim \mathfrak D$ and $\hocolim \mathfrak F: \hocolim \mathfrak C \to \hocolim \mathfrak D$.
\end{defn}

\begin{defn}[Complex for a covering]
Let $\mathscr A = \{A_i\}_{i \in I}$ be a covering of a space $X$, and let $\mathcal N(\mathscr A)'$ denote the barycentric subdivision of the nerve $\mathcal N(\mathscr A)$. This is a simplicial complex with a vertex $v_\sigma$ for each $\sigma \subset I$ such that $A_\sigma \ne \emptyset$, and has the structure of a $\Delta$-complex given by ordering the vertices $v_\sigma < v_\tau$ whenever $\sigma \supset \tau$. We define a complex of spaces $\mathfrak C_{\mathscr A}$ over $\mathcal N(\mathscr A)'$ by setting $\mathfrak C_{\mathscr A}(v_\sigma) \coloneqq A_\sigma$, and letting $\mathfrak C_{\mathscr A}[v_\sigma,v_\tau]$ be the inclusion $A_\sigma \hookrightarrow A_\tau$.
\end{defn}

The following lemma is a generalization of Proposition 5.3 of \cite{gonzalez-gonzalez}, which is itself a generalization of Proposition 4G.2 of \cite{hatcher}. We use the refinement $\mathscr A$ of $\mathscr U$ in order to construct a $G$-invariant partition of unity subordinate to $\mathscr U$.

\begin{lem}\label{partition of unity}
Let $G$ be a discrete group acting properly discontinuously on a paracompact Hausdorff space $X$. Let $\mathscr U = \{U_i\}_{i \in I}$ be a locally finite $G$-invariant open covering with an equivalent $G$-invariant closed refinement $\mathscr A = \{A_i\}_{i \in I}$. If each $G_i$ is finite, then $p_f: \hocolim \mathfrak C_{\mathscr U} \to \colim \mathfrak C_{\mathscr U}=X$ is a $G$-homotopy equivalence.
\end{lem}

\begin{proof}
We first construct a $G$-invariant partition of unity subordinate to $\mathscr U$. We take a transversal of the $G$-action: let $J \subset I$ be such that for each $G$-orbit $GA_i$ of the action $G \times \mathscr A \to \mathscr A$, there is a unique $j \in J$ with $A_j \subset GA_i$. For each $j \in J$, let us take an open set $U_j'$ satisfying
\[
A_j \subset U_j' \subset \overline{U_j'} \subset U_j.
\]

Since the space $X$ is normal, we may apply Urysohn's Lemma to obtain for each $j \in J$ a function $\eta_j:X \to [0,1]$ satisfying $\eta_j(X \setminus U_j') = 0$ and $\eta_j(A_j) = 1$. Since $G_j$ is finite, we define the finite sum
\[
\psi_j \coloneqq \sum_{g \in G_j} \eta_j \circ g.
\]
The function $\psi_j$ satisfies $\psi_j \circ g = \psi_j$ for every $g \in G_j$, and its support is contained in $\bigcup_{g \in G_j} g(\overline{U_j'}) \subset U_j$. Now, for each $i \in I$ with $g(A_i) = A_j$, we set $\psi_i \coloneqq \psi_j \circ g$. Thus for every $i \in I$, we have $A_i \subset \supp \psi_i \subset U_i$. Since $\mathscr U$ is locally finite, the sum $\sum_{i \in I} \psi_i$ is finite, and since $\mathscr A$ covers $X$, this sum is everywhere nonzero. By construction, this sum is $G$-invariant.

Setting $\phi_i = \psi_i/\sum_{i \in I} \psi_i$, we conclude that $\Phi \coloneqq \{\phi_i\}_{i \in I}$ is a $G$-invariant partition of unity subordinate to $\mathscr U$. The remainder of the proof of this lemma is now identical to the proof of Proposition 5.3 of \cite{gonzalez-gonzalez}. We therefore omit some routine verifications.

The formula
\[
\ell_\Phi(x) \coloneqq \left[\left(\sum_i \phi_i(x) \cdot u_i, x \right)\right]
\]
gives a $G$-equivariant map $X \to \hocolim \mathfrak C_{\mathscr U}$ that clearly satisfies $p_f \circ \ell_\Phi = \mathrm{Id}_X$. It remains to show that $\ell_\Phi \circ p_f$ is $G$-homotopic to $\mathrm{Id}_{\hocolim \mathfrak C_{\mathscr U}}$. It is routine to verify that the linear deformation
\[
H\left(\left[\left(\sum_i t_i \cdot u_i, x\right)\right], t\right) = \left[\left(\sum_i (tt_i + (1-t)\phi_i(x)) \cdot u_i, x\right)\right]
\]
is the desired homotopy.
\end{proof}

\begin{lem}[Proposition 5.5 of \cite{gonzalez-gonzalez}]\label{equivariant homotopy lemma}
Given a map $\mathfrak F: \mathfrak C \to \mathfrak D$ of $G$-complexes of spaces over a $\Delta$-complex $B$, if $\mathfrak F(v)$ is a $G_v$-homotopy equivalence for every vertex $v$ of $B$, then $\hocolim \mathfrak F: \hocolim \mathfrak C \to \hocolim \mathfrak D$ is a $G$-homotopy equivalence.
\end{lem}

While \cite{gonzalez-gonzalez} is concerned primarily with the case where $G$ is finite, it is straightforward to check that Lemma \ref{equivariant homotopy lemma} holds just as well for discrete groups with all $G_v$ finite.

\begin{proof}[Proof of Theorem \ref{equivariant nerve lemma}]
Let $*_{\mathcal N(\mathscr A)'}$ be the complex of spaces over $\mathcal N(\mathscr A)'$ where $*_{\mathcal N(\mathscr A)'}(v_\sigma)$ is a single point $*$ for every $\sigma \subset I$. Observe that $\hocolim *_{\mathcal N(\mathscr A)'} = \mathcal N(\mathscr A)'$. For each $\sigma \subset I$, we have a $G_\sigma$-deformation retraction $\mathfrak F(v_\sigma):A_\sigma \to *$, and hence it follows from Lemma \ref{equivariant homotopy lemma} that $\hocolim \mathfrak C_{\mathscr A} \simeq_G \mathcal N(\mathscr A)'$. Since $\mathscr A$ is an equivalent refinement of $\mathscr U$, we also have $G_\sigma$-deformation retractions $U_\sigma \to A_\sigma$, and so Lemma \ref{equivariant homotopy lemma} again gives $\hocolim \mathfrak C_{\mathscr U} \simeq_G \hocolim \mathfrak C_{\mathscr A}$. Finally, Lemma \ref{partition of unity} gives $\hocolim \mathfrak C_{\mathscr U} \simeq_G X$. Altogether, we have $X \simeq_G \mathcal N(\mathscr A)'$. Since $\mathcal N(\mathscr A)$ and $\mathcal N(\mathscr A)'$ are equal as topological spaces, we are done.
\end{proof}

\bibliography{}{}
\bibliographystyle{amsalpha}

\end{document}